\documentclass[12pt,reqno,a4paper]{amsart}

\def\langform{ja}
\def\ja{ja}
\newcommand{\lang}[1]{#1}
\newcommand{\kurz}[1]{}
\newcommand{\bmath}{\kurz{\begin{math}}\lang{\begin{displaymath}} }
\newcommand{\emath}{\kurz{\end{math}}\lang{\end{displaymath}} }

\usepackage[latin1]{inputenc}
\usepackage{amsmath}
\usepackage{amsthm}
\usepackage{amssymb}
\usepackage{amsfonts}
\usepackage{enumerate}
\usepackage{stmaryrd}
\usepackage{enumitem}
\usepackage{centernot}
\usepackage{todonotes}
\usepackage{mathrsfs}
\usepackage{listings}
\usepackage{multicol}
\usepackage[noend]{algorithmic} % for the algorithmic enviroment

\usepackage{array}
\newcolumntype{L}[1]{>{\raggedright\let\newline\\\arraybackslash\hspace{0pt}}m{#1}}
\newcolumntype{C}[1]{>{\centering\let\newline\\\arraybackslash\hspace{0pt}}m{#1}}
\newcolumntype{R}[1]{>{\raggedleft\let\newline\\\arraybackslash\hspace{0pt}}m{#1}}

\usepackage{varwidth}

\usepackage{tikz}
\usetikzlibrary{arrows}
\usetikzlibrary{decorations.pathmorphing}
\usetikzlibrary{decorations.pathreplacing}
\usetikzlibrary{matrix}
\usetikzlibrary{calc}
\usetikzlibrary{shapes}
\usetikzlibrary{patterns}
\usetikzlibrary{fit,backgrounds,scopes}

\makeatletter
\pgfdeclarepatternformonly[\LineSpace,\LineWidth]{my horizontal lines}%
{\pgfpointorigin}{\pgfqpoint{100pt}{1pt}}{\pgfqpoint{100pt}{\LineSpace}}%
{
\pgfsetcolor{\tikz@pattern@color}
\pgfsetlinewidth{\LineWidth}
\pgfpathmoveto{\pgfqpoint{0pt}{0.5pt}}
\pgfpathlineto{\pgfqpoint{100pt}{0.5pt}}
\pgfusepath{stroke}
}

\pgfdeclarepatternformonly[\LineSpace,\LineWidth]{my vertical lines}%
{\pgfpointorigin}{\pgfqpoint{1pt}{100pt}}{\pgfqpoint{\LineSpace}{100pt}}%
{
\pgfsetcolor{\tikz@pattern@color}
\pgfsetlinewidth{\LineWidth}
\pgfpathmoveto{\pgfqpoint{0.5pt}{0pt}}
\pgfpathlineto{\pgfqpoint{0.5pt}{100pt}}
\pgfusepath{stroke}
}

\pgfdeclarepatternformonly[\LineSpace,\LineWidth]{my grid}%
{\pgfqpoint{-1pt}{-1pt}}{\pgfqpoint{\LineSpace}{\LineSpace}}
{\pgfqpoint{\LineSpace}{\LineSpace}}%
{
\pgfsetcolor{\tikz@pattern@color}
\pgfsetlinewidth{\LineWidth}
\pgfpathmoveto{\pgfqpoint{0pt}{0pt}}
\pgfpathlineto{\pgfqpoint{0pt}{\LineSpace + 0.1pt}}
\pgfpathmoveto{\pgfqpoint{0pt}{0pt}}
\pgfpathlineto{\pgfqpoint{\LineSpace + 0.1pt}{0pt}}
\pgfusepath{stroke}
}

\pgfdeclarepatternformonly[\LineSpace,\LineWidth]{my north east lines}
{\pgfqpoint{-\LineWidth}{-\LineWidth}}{\pgfqpoint{\LineSpace}{\LineSpace}}
{\pgfqpoint{\LineSpace}{\LineSpace}}%
{
\pgfsetcolor{\tikz@pattern@color}
\pgfsetlinewidth{\LineWidth}
\pgfpathmoveto{\pgfqpoint{-\LineWidth}{-\LineWidth}}
\pgfpathlineto{\pgfqpoint{\LineSpace + 0.1pt}{\LineSpace + 0.1pt}}
\pgfusepath{stroke}
}

\pgfdeclarepatternformonly[\LineSpace,\LineWidth]{my north west lines}
{\pgfqpoint{-\LineWidth}{-\LineWidth}}{\pgfqpoint{\LineSpace}{\LineSpace}}
{\pgfqpoint{\LineSpace}{\LineSpace}}%
{
\pgfsetcolor{\tikz@pattern@color}
\pgfsetlinewidth{\LineWidth}
\pgfpathmoveto{\pgfqpoint{-\LineWidth}{\LineSpace}}
\pgfpathlineto{\pgfqpoint{\LineSpace + 0.1pt}{-\LineWidth}}
\pgfusepath{stroke}
}

\pgfdeclarepatternformonly[\LineSpace,\LineWidth]{my crosshatch}%
{\pgfqpoint{-1pt}{-1pt}}{\pgfqpoint{\LineSpace}{\LineSpace}}
{\pgfqpoint{\LineSpace}{\LineSpace}}%
{
\pgfsetcolor{\tikz@pattern@color}
\pgfsetlinewidth{\LineWidth}
\pgfpathmoveto{\pgfqpoint{\LineSpace + 0.1pt}{0pt}}
\pgfpathlineto{\pgfqpoint{0pt}{\LineSpace + 0.1pt}}
\pgfpathmoveto{\pgfqpoint{0pt}{0pt}}
\pgfpathlineto{\pgfqpoint{\LineSpace + 0.1pt}{\LineSpace + 0.1pt}}
\pgfusepath{stroke}
}

\pgfdeclarepatternformonly[\LineSpace,\PointSize]{my dots}%
{\pgfqpoint{-\LineSpace*0.25}{-\LineSpace*0.25}}
{\pgfqpoint{\LineSpace*0.25}{\LineSpace*0.25}}
{\pgfqpoint{\LineSpace*0.75}{\LineSpace*0.75}}%
{
\pgfsetcolor{\tikz@pattern@color}
\pgfpathcircle{\pgfqpoint{0pt}{0pt}}{\PointSize}
\pgfusepath{fill}
}
\makeatother

\newdimen\LineSpace
\newdimen\PointSize
\newdimen\LineWidth
\tikzset{
line space/.code={\LineSpace=#1},
line space=3pt
}
\tikzset{
point size/.code={\PointSize=#1},
point size=.5pt
}
\tikzset{
pattern line width/.code={\LineWidth=#1},
pattern line width=.4pt
}
\newtheoremstyle{theoremstyle}
  {10pt}      %  Space above
  {5pt}       %  Space below
  {\itshape}  %  Body font
  {}          %  Indent amount (empty = no indent, \parindent = para indent)
  {\bfseries} %  Thm head font
  {}         %  Punctuation after thm head
  {\newline}      %  Space after thm head: " " = normal interword space;
  {}          %  Thm head spec (can be left empty, meaning `normal')

\newtheoremstyle{examplestyle}
  {10pt}      %  Space above
  {5pt}       %  Space below
  {}          %  Body font
  {}          %  Indent amount (empty = no indent, \parindent = para indent)
  {\bfseries} %  Thm head font
  {}         %  Punctuation after thm head
  {\newline}      %  Space after thm head: " " = normal interword space;
  {}          %  Thm head spec (can be left empty, meaning `normal')

\theoremstyle{theoremstyle}
\newtheorem{theorem}{Theorem}[section]
\newtheorem*{theorem*}{Theorem}
\newtheorem{lemma}[theorem]{Lemma}
\newtheorem{proposition}[theorem]{Proposition}
\newtheorem*{proposition*}{Proposition}
\newtheorem{corollary}[theorem]{Corollary}
\newtheorem*{corollary*}{Corollary}

\theoremstyle{examplestyle}
\newtheorem{example}[theorem]{Example}
\newtheorem{definition}[theorem]{Definition}
\newtheorem{definition*}{Definition}

\newtheorem{remark}[theorem]{Remark}
\newtheorem{remark*}{Remark}
\newtheorem{convention}[theorem]{Convention}

\newtheorem{algorithm}[theorem]{Algorithm}

\newcommand{\NN }{\mathbb{N}}

\newcommand{\RR }{\mathbb{R}}
\newcommand{\QQ }{\mathbb{Q}}
\newcommand{\ZZ }{\mathbb{Z}}

\newcommand{\ut}{{\underline{t}}}
\newcommand{\ux}{{\underline{x}}}

\newcommand{\Rtx}{{R\llbracket t\rrbracket [x]}}
\newcommand{\Rtxx}{{R\llbracket t\rrbracket [x_h]}}

\newcommand{\suchthat}{\;\ifnum\currentgrouptype=16 \middle\fi|\;}
\newcommand{\bigmid}{\left.\vphantom{\Big\{} \suchthat \vphantom{\Big\}}\right.}

\newcommand{\skipalgorithm}{\ \vspace*{-3ex}}
\DeclareMathOperator{\syz}{syz}

\DeclareMathOperator{\Mon}{Mon}

\DeclareMathOperator{\lm}{LM}
\DeclareMathOperator{\lc}{LC}
\DeclareMathOperator{\lt}{LT}

\DeclareMathOperator{\LM}{LM}
\DeclareMathOperator{\LT}{LT}

\DeclareMathOperator{\tail}{tail}
\DeclareMathOperator{\ecart}{ecart}

\DeclareMathOperator{\HDDwR}{HDDwR}
\DeclareMathOperator{\SHDDwR}{SHDDwR}

\DeclareMathOperator{\pRed}{pRed}
\DeclareMathOperator{\DwR}{DwR}
\DeclareMathOperator{\SDwR}{SDwR}
\DeclareMathOperator{\Div}{Div}

\DeclareMathOperator{\NF}{NF}

\DeclareMathOperator{\lcm}{lcm}
\DeclareMathOperator{\spoly}{spoly}

\begin{document}

   \parindent0cm

   \title[Standard Bases]{Standard Bases in mixed Power Series and Polynomial
     Rings over Rings}
   \author{Thomas Markwig}
   \address{Technische Universit\"at Kaiserslautern\\
     Fachbereich Mathematik\\
     Erwin--Schr\"odinger--Stra\ss e\\
     D --- 67663 Kaiserslautern
     }
   \email{keilen@mathematik.uni-kl.de}
   \urladdr{http://www.mathematik.uni-kl.de/\textasciitilde keilen}

   \author{Thomas Markwig}
   \address{Technische Universit\"at Kaiserslautern\\
     Fachbereich Mathematik\\
     Erwin--Schr\"odinger--Stra\ss e\\
     D --- 67663 Kaiserslautern
     }
   \email{ren@mathematik.uni-kl.de}

   \author{Oliver Wienand}
   \address{Technische Universit\"at Kaiserslautern\\
     Fachbereich Mathematik\\
     Erwin--Schr\"odinger--Stra\ss e\\
     D --- 67663 Kaiserslautern
     }
   \email{wienand@mathematik.uni-kl.de}

   \thanks{The first and the second author were supported by the German Israeli Foundation
     grant no 1174-197.6/2011 and by the German Research Foundation
     (Deutsche Forschungsgemeinschaft (DFG)) trough the Priority
     Programme 1489.}


   \date{March, 2015}


   \begin{abstract}
     In this paper we study standard bases for submodules of a mixed
     power series and polynomial ring
     $R\llbracket t_1,\ldots,t_m\rrbracket[x_1,\ldots,x_n]^s$ respectively of their
     localization with respect to a $\ut$-local
     monomial ordering for a certain class of noetherian rings
     $R$. The main steps are to prove the existence of
     a division with remainder generalizing and combining the
     division theorems of Grauert--Hironaka and Mora and to
     generalize the Buchberger criterion. Everything else then
     translates naturally. Setting either $m=0$ or $n=0$ we get
     standard bases for polynomial rings respectively for power series
     rings over $R$ as a special case.
%      We then introduce the notion of
%      initially reduced standard basis in the case $m=1$ and give an
%      algorithm to compute initially reduced standard bases in finite
%      time for special ideals which come up naturally when computing
%      tropical varieties over non-archimedian valued fields.
   \end{abstract}

   \maketitle

   The paper follows to a large part the lines of \cite{Mar08}, or
   alternatively  \cite{GP02} and \cite{DS07},
   adapting to the situation that the coefficient domain $R$ is no field.
   We generalize the Division Theorem
   of Grauert--Hironaka respectively Mora (the latter in the form
   stated and proved first by Greuel and Pfister, see
   \cite{GGMNPPSS94}, \cite{GP96}; see
   also \cite{Mor82}, \cite{Gra94}). The paper should
   therefore be seen as a unified approach for the existence
   of standard bases in polynomial and power series rings for
   coefficient domains which are not fields. Standard bases of ideals
   in such rings come up naturally when computing Gröbner fans (see
   \cite{MR15a}) and
   tropical varieties  (see \cite{MR15b}) over non-archimedian valued fields, even though
   we consider a wider class of base rings than actually needed for
   this.

   An important point is that if the input data is
   polynomial in both $\ut$ and $\ux$ then we can actually compute the
   standard basis in finite time since a standard basis computed in $R[t_1,\ldots,t_m]_{\langle
     t_1,\ldots,t_m\rangle}[x_1,\ldots,x_n]$ will do.% (see Corollary \ref{cor:polynomialcase}).

%    In this paper we treat only formal power series, while Grauert (see
%    \cite{Gra72}) and Hironaka (see \cite{Hir64}) considered convergent
%    power series with respect to certain valuations which includes
%    the formal case. It should be rather straightforward how to adjust
%    Theorem \ref{thm:HDDwR} accordingly.
   Many authors contributed to
   the further development (see e.g.\ \cite{Bec90} for a standard
   basis criterion in the power series ring) and to generalizations of the theory, e.g.\
   to algebraic power series (see e.g.\ \cite{Hir77}, \cite{AMR92}, \cite{ACH05})
   or to differential operators (see e.g.\ \cite{GH05}). This list is
   by no means complete.

%    In Section \ref{sec:basicnotion} we introduce the basic notions.
%    Section \ref{sec:HDDwR} is devoted to the proof of the existence of
%    a determinate division with remainder for polynomials
%    in $K\llbracket t_1,\ldots,t_m\rrbracket[x_1,\ldots,x_n]^s$ which are homogeneous with
%    respect to the $x_i$. This result is then used in Section
%    \ref{sec:DwR} to show the existence of weak divisions with
%    remainder for all elements of
%    $K\llbracket t_1,\ldots,t_m\rrbracket[x_1,\ldots,x_n]^s$. In Section
%    \ref{sec:standardbases} we introduce standard bases and prove
%    the basics for these, and we prove Schreyer's Theorem and, thus
%    Buchberger's Criterion in  Section \ref{sec:schreyer}.
%    \lang{In Section \ref{sec:algorithms} we describe some algorithms
%      which rely on the standard basis algorithm, and if the input is
%      polynomial in $\ut$ as well as in $\ux$ then the algorithms
%      terminate.}
%    Finally,
%    in Section \ref{sec:application} we apply standard bases to study
%    $t$-initial ideals of ideals over the Puiseux series field.

%%%%%%%%%%%%%%%%%%%%%%%%%%%%%%%%%%%%%%%%%%%%%%%%%%%%%%%%%%%%%%%%%%%%%%%%%%%%%%%%%%

   \section{Division with remainder}

   In this section, we construct a division with remainder following
   the first three chapters of \cite{Markwig08}. Please mind the
   assumptions on our ground ring in Convention \ref{con:groundRing} for
   that, which were taken from Definition 1.3.14 in \cite{Wienand11}.

   After a quick introduction of the basic terminology, we begin with
   a division algorithm over the ground ring in the form of
   Algorithm~\ref{alg:coeff}. We then continue with homogeneous division
   with remainder in Algorithm~\ref{alg:HDDwR}, and finally end with a
   weak division with remainder in Algorithm~\ref{alg:DwR}.

   \begin{convention}[The class of base rings]\label{con:groundRing} For this chapter, let $R$
     be a noetherian ring in which \emph{linear equations are solvable} as in
     Definition 1.3.14 of \cite{Wienand11}. The latter means that, given
     any finite tuple of arbitrary length $(c_1,\ldots,c_k)$ with $c_i \in
     R$, we must be able to do the following:
     \begin{enumerate}[leftmargin=*]
     \item decide for $b\in R$ whether $b\in\langle
       c_1,\ldots,c_k\rangle$, and, if yes, find $a_1,\ldots,a_k\in R$ such
       that
       \begin{displaymath}
 b = a_1\cdot c_1+\cdots+a_k\cdot c_k.
\end{displaymath}
     \item find a finite generating set $S\subseteq R^k$ of its
       syzygies as module over $R$,
       \begin{displaymath}
 \syz_R(c_1,\ldots,c_k)=\{(a_1,\ldots,a_k)\in R^k \mid
       a_1\cdot c_1+\ldots+a_k\cdot c_k = 0\} = \langle S \rangle_R.
\end{displaymath}
     \end{enumerate}
     We will use the notion
     $\Rtx:=R\llbracket t_1,\ldots,t_m\rrbracket[x_1,\ldots,x_n]$ to denote
     a mixed power series and polynomial ring over $R$ in several variables
     $t=(t_1,\ldots,t_m)$ and $x=(x_1,\ldots,x_n)$, and  $\Rtx^s$ will
     denote the free module of rank $s$ over $\Rtx$.

     $R$ being noetherian is most notably required for the conditional
     termination of Algorithm~\ref{alg:DwR}, while linear equations being
     solvable is required in the instructions of Algorithm~\ref{alg:coeff}
     and Algorithm~\ref{alg:standardBases}.
   \end{convention}

   \begin{example}\label{ex:groundRing}
     Admissible ground rings satisfying Convention
     \ref{con:groundRing} include the following:
     \begin{itemize}[leftmargin=*]
     \item Obviously any field, assuming we are able to compute inverse elements.
     \item The ring of integers $\ZZ$. The division with remainder in
       $\ZZ$ allows us to solve the ideal membership problem, while
       the least common multiple allows us to compute finite
       generating sets of syzygies, see Theorem 2.2.5 in
       \cite{Wienand11} for the latter.
     \item Also, $\ZZ/m\ZZ$ for an arbitrary $m\in\ZZ$. While it
       generally is neither Euclidean nor factorial like $\ZZ$, many
       problems can nonetheless be solved by tracing them back to the
       integers.
     \item Similarly, any Euclidean ring for which we are able to
       compute its division with remainder, or, more generally, any
       factorial ring for which we can compute the unique
       factorization. Classical examples hereof are the ring of
       Gaussian integers $\ZZ[i]$, the polynomial ring
       $\QQ[y]$, the power series ring $\QQ\llbracket s \rrbracket$ or multivariate
       polynomial rings.
     \item Moreover, thanks to the theory of Gr\"obner bases, any quotient ring of a polynomial ring,
       e.g. the ring of Laurent polynomials $K[y_1^{\pm
         1},\ldots,y_n^{\pm 1}]=K[y_0,\ldots,y_n]/(1-y_0\cdots y_n)$.
     \item And, thanks to the theory of standard bases, any
       localization of a polynomial ring at a prime ideal, as it can
       be traced back to a quotient of a polynomial ring localized at
       a mixed ordering, see \cite{Mora91}.
     \item Also, Dedeking domains. A solution to the ideal membership problem and the computation of syzygies can be found in \cite{KY10}.
     \item Finally, product rings like $\ZZ\times\ZZ$, because any ideal in it is the product of two ideals in $\ZZ$.
     \end{itemize}
   \end{example}

   We now begin with introducing some very basic notions of standard
   basis theory to our ring resp. module, definitions such as
   monomials, monomial orderings and leading monomials.

   \begin{definition}\label{def:basicsRing}
     The \emph{set of monomials} of $\Rtx$ is defined to be
     \begin{displaymath}
       \Mon(t,x):=\{ t^\beta x^\alpha \mid \beta\in\NN^m, \alpha\in\NN^n\}\subseteq\Rtx,
     \end{displaymath}
     and a \emph{monomial ordering} on $\Mon(t,x)$ is an ordering $>$ that is compatible with its natural semigroup structure, i.e.
     \begin{displaymath}
       \forall a,b,q \in\Mon(t,x): \quad a> b \quad \Longrightarrow \quad q\cdot a > q\cdot b.
     \end{displaymath}
     We call a monomial ordering $>$ \emph{$t$-local}, if $1 >t^\beta$ for all $\beta\in\NN^m$.

     Let $>$ be a $t$-local monomial ordering on $\Mon(t,x)$, and let
     $w\in\RR_{<0}^m\times\RR^{n}$ be a weight vector. Then the
     ordering $>_w$ is defined to be:
     \begin{align*}
       t^{\beta}x^\alpha >_w t^{\delta}x^\gamma\cdot \quad
       :\Longleftrightarrow \quad & w\cdot(\beta,\alpha) >
       w\cdot(\delta,\gamma) \text{ or } \\
       & w\cdot(\beta,\alpha) = w\cdot(\delta,\gamma) \text{ and } t^{\beta}x^\alpha > t^{\delta}x^\gamma.
     \end{align*}
     We will refer to orderings of the form $>_w$ as a \emph{weighted ordering} with weight vector $w$ and tiebreaker $>$.
   \end{definition}

   \begin{definition}\label{def:basicsModule}
     The \emph{set of module monomials} of $\Rtx^s$ is defined to be
     \begin{displaymath}
       \Mon^s(t,x):=\{ t^\beta x^\alpha \cdot e_i \mid \beta\in\NN^m, \alpha\in\NN^n, i=1,\ldots,s\} \subseteq \Rtx^s.
     \end{displaymath}
     A \emph{monomial ordering} on $\Mon^s(t,x)$ is an ordering $>$ that is compatible with the natural $\Mon(t,x)$-action on it, i.e.
     \begin{displaymath}
       \forall a,b\in \Mon^s(t,x)\;\;  \forall q\in\Mon(t,x): \quad a > b \quad \Longrightarrow \quad q\cdot a > q\cdot b,
     \end{displaymath}
     and that restricts onto the same monomial ordering on $\Mon(t,x)$ in each component, i.e.
     \begin{displaymath}
       \forall a,b\in\Mon(t,x) \;\;\forall i,j\in\{1,\ldots,s\}: \quad
       a\cdot e_i > b\cdot e_i \quad \Longleftrightarrow \quad a\cdot
       e_j > b\cdot e_j.
     \end{displaymath}
     We call a monomial ordering $>$ \emph{$t$-local}, if $1\cdot e_i>t^\beta\cdot e_i$ for all $\beta\in\NN^m$ and $i=1,\ldots,s$.

     Let $>$ be a $t$-local monomial ordering on $\Mon^s(t,x)$, and
     let $w\in\RR_{<0}^m\times\RR^n\times\RR^s$ be a weight
     vector. Then the ordering $>_w$ is defined to be:
     \begin{align*}
       & t^{\beta}x^\alpha \cdot e_i >_w t^{\delta}x^\gamma\cdot e_j \quad \Longleftrightarrow\\
       & \quad w\cdot(\beta,\alpha,e_i) > w\cdot(\delta,\gamma,e_j) \text{ or }\\
       & \quad w\cdot(\beta,\alpha,e_i) = w\cdot(\delta,\gamma,e_j) \text{ and } t^{\beta}x^\alpha \cdot e_i > t^{\delta}x^\gamma\cdot e_j.
     \end{align*}
     We will refer to orderings of the form $>_w$ as a \emph{weighted ordering} with weight vector $w$ and tiebreaker $>$.

     From now on, we will simply refer to module monomials as monomials.
   \end{definition}

   \begin{definition}
     Given a $t$-local monomial ordering $>$ on $\Mon^s(t,x)$ and an
     element $f=\sum_{\alpha,\beta,i} c_{\alpha,\beta,i} \cdot t^\beta
     x^\alpha\cdot  e_i \in\Rtx^s$,
     we define its \emph{leading monomial}, \emph{leading coefficient}, \emph{leading term} and \emph{tail} to be
     \begin{align*}
       \lm_>(f) &= \max\{ t^\beta x^\alpha \cdot e_i \mid c_{\alpha,\beta,i}\neq 0\}, \\
       \lc_>(f) &= c_{\alpha,\beta,i}, \text{ where } t^{\beta}x^{\alpha}\cdot e_i = \lm_>(f), \\
       \lt_>(f) &= c_{\alpha,\beta,i} \cdot t^{\beta}x^{\alpha}\cdot e_i, \text{ where } t^{\beta}x^{\alpha}\cdot e_i = \lm_>(f), \\
       \tail_>(f) &= f-\lt_>(f).
     \end{align*}
     For a submodule $M\leq \Rtx^s$, we set
     \begin{align*}
       \lm_>(M) &= \langle \lm_>(f)\mid f\in M\rangle_{R[t,x]}\leq R[t,x]^s, \\
       \lt_>(M) &= \langle \lt_>(f)\mid f\in M \rangle_{R[t,x]}\leq R[t,x]^s.
     \end{align*}
     Note that we regard the two modules above as submodules of $R[t,x]^s$, while the original module lies in $\Rtx^s$.
     We refer to $\lt_>(M)$ as the \emph{leading module} of $M$ with respect to $>$.
   \end{definition}

   \begin{example}
     Observe that in general
     \begin{displaymath}
       \LM_>(M)\neq \LT_>(M).
     \end{displaymath}
     Consider the ideal
     \begin{displaymath}
       I:=\langle 1+t^6x+t^4y+t^7x^2+t^5xy+t^8y^2, 2-t\rangle \unlhd \ZZ\llbracket t \rrbracket [x],
     \end{displaymath}
     and let $>_w$ be the weighted ordering with
     weight vector $w=(-1,3,3)$ and any arbitrary tiebreaker. Then by
     weighted degree alone we have
     \begin{displaymath}
       \LT_{>_w}(I)=\langle t^5xy,2\rangle\neq \LM_{>_w}(I)=\langle 1\rangle,
     \end{displaymath}
     since  $\LM_{>_w}(2-t)=1$. In fact, the last equation holds true for any $t$-local monomial
     ordering, while the former varies depending on the ordering. This
     is why the role of leading monomials in the classical standard
     basis theory over fields is played by leading terms over rings.
   \end{example}

   \begin{remark}\label{rem:tLocality}
     Note that the $t$-locality of the monomial ordering $>$ is
     essential for leading monomials and other associated objects to
     exist, as elements of $\Rtx$ resp. $\Rtx^s$ may be unbounded in
     their degrees of $t$.

     However, given a weight vector in $\RR_{<0}^m\times\RR^{n}$
     resp. $\RR_{<0}^m\times\RR^n\times\RR^s$, a weighted monomial
     ordering does not need a $t$-local tiebreaker for leading
     monomials to be well-defined. But for sake of simplicity, we
     nevertheless assume all occuring monomial orderings to be
     $t$-local.
   \end{remark}

   $\Mon(t,x)$ comes equipped with a natural notion of divisibility and least common multiple. For module monomials, we define:

   \begin{definition}\label{def:divisibilityAndLCM}
     For two module monomials $t^\beta x^\alpha\cdot e_i$ and $t^\delta x^\gamma\cdot e_j\in\Mon(t,x)^s$, we say
     \begin{displaymath}
       t^\beta x^\alpha\cdot e_i \text{ divides } t^\delta
       x^\gamma\cdot e_j \quad :\Longleftrightarrow \quad e_i=e_j
       \text{ and } t^\beta x^\alpha \text{ divides } t^\delta
       x^\gamma,
     \end{displaymath}
     and in this case we set
     \begin{displaymath}
       \frac{t^\beta x^\alpha\cdot e_i}{t^\delta x^\gamma\cdot e_j}:=\frac{t^\beta x^\alpha}{t^\delta x^\gamma}=t^{\beta-\delta}x^{\alpha-\gamma}\in\Mon(t,x).
     \end{displaymath}
     We define the \emph{least common multiple} of two module
     monomials $t^\beta x^\alpha\cdot e_i$ and $t^\delta x^\gamma\cdot
     e_j\in\Mon(t,x)^s$ to be
     \begin{displaymath}
       \lcm(t^\beta x^\alpha\cdot e_i,t^\delta x^\gamma\cdot e_j):=
       \begin{cases}
         \lcm(t^\beta x^\alpha,t^\delta x^\gamma)\cdot e_j, & \text{if } i=j, \\
         0,                                                 & \text{otherwise}.
       \end{cases}
     \end{displaymath}
   \end{definition}

   We now devote the remaining section to proving the existence of a division with remainder, starting with its definition.

   \begin{definition}\label{def:dwr}
     Let $>$ be a $t$-local monomial ordering on $\Mon^s(t,x)$.
     Given $f\in\Rtx^s$ and $g_1,\ldots,g_k\in\Rtx^s$ we say that a representation
     \begin{displaymath}
       f=q_1\cdot g_1+\ldots+q_k\cdot g_k+r
     \end{displaymath}
     with $q_1,\ldots,q_k\in\Rtx$ and $r=\sum_{j=1}^s r_j\cdot e_j\in\Rtx^s$ satisfies
     \begin{description}
     \item[\rm (ID1)] if $\lm_>(f)\geq\lm_>(q_i\cdot g_i)$ for all $i=1,\ldots,k$,
     \item[\rm (ID2)] if $\lt_>(r)\notin \langle \lt_>(g_1),\ldots,\lt_>(g_k)\rangle$, unless $r=0$,
     \item[\rm (DD1)] if no term of $q_i\cdot \lt_>(g_i)$ lies in
       $\langle \lt_>(g_j) \mid j<i \rangle$ for all $i=1,\ldots,k$,
     \item[\rm (DD2)] if no term of $r$ lies in $\langle \lt_>(g_1),\ldots,\lt_>(g_k) \rangle$,
     \item[\rm (SID2)] if $\lt_>(r_j\cdot e_j)\notin\langle \lt_>(g_1),\ldots,\lt_>(g_k) \rangle$, unless $r_j=0$, for all $j=1,\ldots,s$.
     \end{description}
     A representation satisfying (ID1) and (ID2) is called an
     \emph{(indeterminate) division with remainder}, and a
     representation satisfying (DD1) and (DD2) is called a
     \emph{determinate division with remainder}. In each of these two cases
     we call $r$ a \emph{remainder} or \emph{normal form} of $f$ with
     respect to $(g_1,\ldots,g_k)$. Moreover, if the remainder $r$ is
     zero, we call the representation a \emph{standard representation}
     of $f$ with respect to $(g_1,\ldots,g_k)$.

     A division with remainder of $u\cdot f$ for some $u\in \Rtx$ with
     $\lt_>(u)=1$ is also called a \emph{weak division with remainder}
     of $f$. A remainder of $u\cdot f$ will be called a \emph{weak
       normal form} of $f$ with respect to $(g_1,\ldots,g_k)$, and a
     standard representation of $u\cdot f$ will be called a \emph{weak
       standard representation} of $f$.
   \end{definition}

   \begin{proposition}\label{prop:dwrs}
     Consider a representation
     \begin{displaymath}
       f=q_1\cdot g_1+\ldots+q_k\cdot g_k+r \quad \text{or} \quad u\cdot f=q_1\cdot g_1+\ldots+q_k\cdot g_k+r
     \end{displaymath}
     with $f,g_1,\ldots,g_k,r\in\Rtx^s$, $q_1,\ldots,q_k\in\Rtx$ and $\lt_>(u)=1$. Then:
     \begin{enumerate}[leftmargin=*]
     \item if the representation satisfies {\rm (DD2)}, then it also satisfies {\rm (SID2)},
     \item if the representation satisfies {\rm (SID2)}, then it also satisfies {\rm (ID2)},
     \item if it satisfies both {\rm (DD1)} and {\rm (ID2)}, then it also satisfies {\rm (ID1)}.
     \end{enumerate}
     In particular, {\rm (DD1)} and {\rm (DD2)} imply {\rm (ID1)} and {\rm (ID2)}.
   \end{proposition}
   \begin{proof}
      \lang{(1) and (2) are obvious, so suppose the representation satisfies both (DD1) and (DD2).

      Take the maximal monomial $t^\beta x^\alpha$ occurring in any of
      the expressions $q_i\cdot g_i$ or $r$ on the right hand side, and
      assume $t^\beta x^\alpha>\lm_>(f)$. Because of maximality, it has
      to be the leading monomial of each expression it occurs in. And
      because it does not occur on the left hand side, the leading
      terms have to cancel each other out. Let $q_{i_1}\cdot
      g_{i_1},\ldots,q_{i_l}\cdot g_{i_l}$ be the $q_i\cdot g_i$
      containing $t^\beta x^\alpha$ with $i_1<\ldots<i_l$.

      If $r$ contains $t^\beta x^\alpha$, then $\sum_{j=1}^l\lt_>(q_{i_j}\cdot g_{i_j}) + \lt_>(r)= 0$, and hence
      \begin{displaymath}
        \lt_>(r)=t^\beta x^\alpha\in \langle \lt_>(g_1),\ldots,\lt_>(g_k)\rangle,
      \end{displaymath}
      contradicting (ID2).

      If $r$ does not contain $a$, then we have $\sum_{j=1}^l\lt_>(q_{i_j}\cdot g_{i_j}) = 0$, thus
      \begin{displaymath}
        \lt_>(q_{i_l}\cdot g_{i_l}) \in \langle \lt_>(g_j) \mid j<i_l \rangle,
      \end{displaymath}
      contradicting (DD1).}
    \kurz{The proof is straight forward.}
   \end{proof}

   Next, we pay a little attention to our ground ring. Convention
   \ref{con:groundRing} states that our ring already comes equipped
   with everything we need to compute representations of members in
   given ideals, but we still need to make sure that these
   representations satisfy our needs in Algorithm~\ref{alg:HDDwR}.

   % \pagebreak
   \begin{algorithm}[$\Div_R$, division in the ground ring]\label{alg:coeff}\skipalgorithm
     \begin{algorithmic}[1]
       \REQUIRE{$(b,C)$, where $C=(c_1,\ldots,c_k)\in R^k$ and $b \in \langle C \rangle$.}
       \ENSURE{$(a_1,\ldots,a_k)\subseteq R^k$, such that \begin{displaymath}
           b=a_1\cdot c_1 + \ldots + a_k\cdot c_k
         \end{displaymath}
         with $a_i\cdot c_i\notin\langle c_j\mid j<i \rangle$ unless $a_i\cdot c_i =0$, for any $i=1,\ldots,k$.}
       \STATE Find $a_1,\ldots,a_k \in R$ with $b=a_1\cdot c_1+\ldots+a_k\cdot c_k$, which is possible by Convention \ref{con:groundRing}.
       \FOR{$i=k,\ldots,1$}
       \IF{$a_i \cdot c_i \neq 0$ and $a_i \cdot c_i \in \langle c_j\mid j<i \rangle$}
       \STATE Find $h_1,\ldots,h_{i-1} \in R$ such that
       $a_i\cdot c_i=h_1\cdot c_1+\ldots+h_{i-1}\cdot c_{i-1}$.
       \STATE Set $a_j:=a_j+h_j$ for all $j<i$, and $a_i:=0$.
       \ENDIF
       \ENDFOR
       \RETURN{$(a_1,\ldots,a_k)$}
     \end{algorithmic}
   \end{algorithm}
   \begin{proof}
     Termination and correctness are obvious.
   \end{proof}

   With this preparation we are able to formulate and prove
   determinate division with remainder for $x$-homogeneous ideals and
   modules.

   \begin{definition}\label{def:homogeneity}
     For an element $f=\sum_{\beta,\alpha,i} c_{\alpha,\beta,i} \cdot
     t^\beta x^\alpha\cdot e_i\in\Rtx^s$ we define its
     \emph{$x$-degree} to be
     \begin{displaymath}
       \deg_x(f):=\max\{|\alpha|\mid c_{\alpha,\beta,i}\neq 0\},
     \end{displaymath}
     and we call it \emph{$x$-homogeneous}, if all its terms are of the same $x$-degree.

     Given a weight vector $w\in\RR_{< 0}^m\times\RR^n\times\RR^s$, we define its \emph{weighted degree} with respect to $w$ to be
     \begin{displaymath}
       \deg_w(f):=\max\{w\cdot(\beta,\alpha,e_i) \mid c_{\alpha,\beta,i}\neq 0\},
     \end{displaymath}
     and we call it \emph{weighted homogeneous} with respect to $w$, if all its terms are of the same weighted degree.
   \end{definition}

   \begin{algorithm}[$\HDDwR$, homogeneous determinate division with remainder]\label{alg:HDDwR}\skipalgorithm
     \begin{algorithmic}[1]
       \REQUIRE{$(f,G,>)$, where $f\in\Rtx^s$ $x$-homogeneous,
         $G=(g_1,\ldots,g_k)$ a $k$-tuple of
         $x$-homogeneous elements in $\Rtx^s$ and $>$ be a $t$-local
         monomial ordering on $\Mon^s(t,x)$.}
       \ENSURE{$(Q,r)$, where $Q=(q_1,\ldots,q_k)\in\Rtx^k$ and $r\in \Rtx^s$ such that
         \begin{displaymath}
           f=q_1\cdot g_1+\hdots+q_k\cdot g_k+r
         \end{displaymath}
         satisfies \kurz{(DD1), (DD2) and}
         \begin{description}
         \lang{\item[\rm (DD1)] no term of $q_i\cdot\lt_>(g_i)$ lies in $\langle\lt_>(g_j)\mid j<i\rangle$ for all $i$,
         \item[\rm (DD2)] no term of $r$ lies in $\langle\lt_>(g_1),\ldots,\lt_>(g_k)\rangle$,}
         \item[\rm (DDH)] the $q_1,\ldots,q_k,r$ are either $0$ or $ x$-homogeneous of $ x$-degree \linebreak
           $\deg_ x(f)-\deg_ x(g_1),\ldots,\deg_ x(f)-\deg_ x(g_k),\deg_ x(f)$ respectively.
         \end{description}}
       \STATE Set $q_i:= 0$ for $i=1,\ldots,k$, $r:=0$, $\nu:=0$, $f_\nu:=f$.
       \WHILE{$f_\nu\neq 0$}
       \IF{$\lt_>(f_\nu) \in \langle\lt_>(g_1),\ldots,\lt_>(g_k)\rangle$}
       \STATE Let $D_\nu:=\{ g_i\in G\mid\lm_>(g_i)\text{ divides }\lm_>(f_\nu) \}\{ g_{i_1},\ldots,g_{i_l} \}$.
       \STATE Compute $(a_{i_1},\ldots,a_{i_l})=\Div_R(\lc_>(f_\nu),(\lc_>(g_{i_1}),\ldots,\lc_>(g_{i_l})))$.
       \STATE Set
       \begin{displaymath}
         q_{i,\nu}:=
         \begin{cases}
           a_i\cdot \frac{\lm_>(f_\nu)}{\lm_>(g_i)} & \text{, if } g_i\in D_\nu, \\
           0                                              & \text{, otherwise, }
         \end{cases}
       \end{displaymath}
       for $i=1,\ldots,k$, and $r_\nu:=0$.
       \ELSE
       \STATE Set $q_{i,\nu}:=0$, for $i=1,\ldots,k$, and $r_\nu:=\lt_>(f_\nu)$.
       \ENDIF
       \STATE Set $q_i:=q_i + q_{i,\nu}$ for $i=1,\ldots,k$ and $r:= r+r_\nu$.
       \STATE Set $f_{\nu+1}:=f_{\nu} - (q_{1,\nu}\cdot g_1 + \ldots + q_{k,\nu}\cdot g_k + r_\nu)$ and $\nu:=\nu+1$.
       \ENDWHILE
       \RETURN{$((q_1,\ldots,q_k),r)$}
     \end{algorithmic}
   \end{algorithm}

   \begin{proof} Note that we have a descending chain of terms to be eliminated
     \begin{displaymath}
       \lm_>(f_0)>\lm_>(f_1)> \lm_>(f_2)> \ldots,
     \end{displaymath}
     which implies that, except the terms that are zero, we have $k+1$ descending chains of factors and remainders
     \begin{center}
       \begin{tabular}{>{$}c<{$} >{$}c<{$} >{$}c<{$} >{$}c<{$} >{$}c<{$} >{$}c<{$} >{$}c<{$} >{$}c<{$}}
         \lm_>(q_{i,0})&>&\lm_>(q_{i,1})&>&\lm_>(q_{i,2})&>& \ldots &, \\
         \lm_>(r_0)    &>&\lm_>(r_1)    &>&\lm_>(r_2)    &>& \ldots &.
       \end{tabular}
     \end{center}

     By construction, each $q_{i,\nu}$, $i=1,\ldots,k$,
     is $ x$-homogeneous of $ x$-degree $\deg_ x(f)-\deg_ x(g_i)$, and
     each $r_\nu$ is $ x$-homogeneous of $ x$-degree $\deg_ x(f)$,
     unless they are zero. Because of Lemma~\ref{lem:markwig1} we may
     assume that the ordering $>$ is a $t$-local weighted monomial ordering.
     Thus, by Lemma \ref{lem:markwig2}, the $q_{i,\nu}$ and $r_\nu$ converge to zero in the
     $\langle t\rangle$-adic topology, so that
     \begin{displaymath}
       q_i:=\sum_{\nu=0}^\infty q_{i,\nu}\in \Rtx \text{ and } r:=\sum_{\nu=0}^\infty r_\nu\in \Rtx^s
     \end{displaymath}
     exist and the following representation satisfies (DDH):
     \begin{equation}\label{eq:HDDwR}
       f=q_1\cdot g_1+\ldots+q_k\cdot g_k+r.
     \end{equation}
     Observe that, because all $q_{i,\nu}$ and $r_\nu$ are terms with
     distinct monomials, each non-zero term of $q_i\cdot \lt_>(g_i)$
     or $r$ equals $q_{i,\nu}\cdot \lt_>(g_i)$ or $r_\nu$
     respectively, for some $\nu\in\NN$.

     So first, let $p$ be a non-zero term of $q_i\cdot \lt_>(g_i)$, say $p=q_{i,\nu}\cdot \lt_>(g_i)$ for some $\nu\in\NN$.
     Then $\lc_>(q_{i,\nu})\neq 0$ implies that
     $\lc_>(q_{i,\nu}\cdot g_i)\notin \langle \lc_>(g_j)\mid j<i \text{ with } g_j\in D_\nu\rangle_R$.
     In particular, we have
     $\lt_>(q_{i,\nu}\cdot g_i)=q_{i,\nu}\cdot \lt_>(g_i)\notin \langle \lt_>(g_j)\mid j<i \text{ with } g_j\in D_\nu\rangle$.
     Therefore we also get $q_{i,\nu}\cdot \lt_>(g_i) \notin \langle \lt_>(g_j)\mid j<i \rangle$,
     since the leading monomials of all $g_j\notin D_\nu$ do not divide $\lm_>(f_\nu)=\lm_>(q_{i,\nu}\cdot g_i)$.
     Thus (\ref{eq:HDDwR}) satisfies (DD1).

     Lastly, let $p$ be a non-zero term of $r$, i.e. $p=r_\nu$ for a suitable $\nu$.
     But because $r_\nu\neq 0$, we have $r_\nu=\lt_>(f_\nu) \notin \langle \lt_>(g_1),\ldots,\lt_>(g_k) \rangle$ by default.
     Therefore, our representation (\ref{eq:HDDwR}) also satisfies (DD2).
   \end{proof}

   In the proof we have used the following two Lemmata whose proof can be
   found in \cite{Markwig08}. The
   first Lemma allows us to restrict ourselves to weighted monomial
   orderings, while the second guarantees $\langle t \rangle$-adic
   convergence.

   \begin{lemma}[\cite{Markwig08} Lemma 2.5]\label{lem:markwig1}
     Let $>$ be a $t$-local monomial ordering on $\Mon^s(t,x)$, and
     let $g_1,\ldots,g_k\in\Rtx^s$ be $x$-homogeneous. Then there
     exists a weight vector $w\in\RR_{<0}^m\times\RR^{n+s}$ such that
     any $t$-local weight ordering with weight vector $w$, say $>_w$, induces the same leading monomials as $>$ on $g_1,\ldots,g_k$, i.e.
     \begin{displaymath}
       \lm_{>_w}(g_i) = \lm_>(g_i) \text{ for all } i=1,\ldots,k.
     \end{displaymath}
   \end{lemma}

\ifx\langform\ja
     \begin{example}\label{ex:markwig1}
       A monomial ordering can always be expressed by an invertible
       matrix. For example, the lexicographical ordering $>$ on
       $\Mon(t,x)$ with $x_1>x_2>1>t$ is given by
       \begin{displaymath}
         t^\beta x^\alpha > t^\delta x^\gamma \quad \Longleftrightarrow
         \quad A\cdot (\beta,\alpha)^t > A\cdot (\delta,\gamma)^t,
         \text{ where
         }A=\left(\begin{smallmatrix}\phantom{-}0&1&0&\\\phantom{-}0&0&1\\-1&0&0 \end{smallmatrix}\right),
       \end{displaymath}
       where the $>$ on the right hand side denotes the lexicographical ordering on~$\RR^3$.

       Consider the polynomial $g=t^5x_1+t^2x_2$.
       In order to find a weight vector $w\in\RR_{<0}\times\RR^2$ such
       that $\lm_{>_w}(g)=\lm_>(g)=t^5x_1$, consider the first row
       vector of $A$, $a_1=(0,1,0)\in\RR^3$. Since
       $a_1\notin\RR_{<0}\times\RR^2$ it represents no viable choice for
       $w$. But because $\deg_{a_1}(t^5x_1) > \deg_{a_1}(t^2x_2)$,
       adding a sufficiently small negative weight in $t$ will not break
       the strict inequality. Hence we obtain
       $w=(-\frac{1}{5},1,0)\in\RR_{<0}\times\RR^2$:
       \begin{center}
         \begin{tikzpicture}
           \matrix (m) [matrix of math nodes, row sep=1.5em, column sep=0em, %text height=1.5ex, text depth=0.25ex
           column 1/.style={anchor=base east},
           column 7/.style={anchor=base west}]
           { \deg_{(0,1,0)}(t^5x_1) & = & 1 & > & \phantom{-}0 & = & \deg_{(0,1,0)}(t^2x_2) \\
             \deg_{(-1/5,1,0)}(t^5x_1) & = & 0 & > & -\frac{2}{5} & = & \deg_{(-1/5,1,0)}(t^2x_2). \\ };
           \draw[->,font=\scriptsize] (m-1-1) -- node[left] {$-(1/5,0,0)$} (m-2-1);
           \draw[->,font=\scriptsize] (m-1-7) -- node[right] {$-(1/5,0,0)$} (m-2-7);
         \end{tikzpicture}
       \end{center}

       In particular, a determinate division with remainder with respect
       to $>_w$ will also be a determinate division with remainder with
       respect to $>$, as (DD1) and (DD2) are only dependant on the
       leading terms.
     \end{example}
\fi

   \begin{lemma}[\cite{Markwig08} Lemma 2.6]\label{lem:markwig2}
     Let $>_w$ be a $t$-local monomial ordering on $\Mon^s(t,x)$ with
     weight vector $w\in\RR_{<0}^m\times\RR^{n+s}$, and let
     $(f_k)_{k\in\NN}$ be a sequence of $x$-homogeneous elements of
     fixed $x$-degree in $\Rtx^s$ such that $\lm_{>_w}(f_k) >
     \lm_{>_w}(f_{k+1})$ for all $k\in\NN$. Then $(f_k)_{k\in\NN}$
     converges to zero in the $\langle t\rangle$-adic topology, i.e.
     \begin{displaymath}
       \forall\, N\in\NN \;\;\exists\, M\in\NN: \;\;f_k\in\langle t\rangle^N\cdot\Rtx^s \;\;\forall k\geq M.
     \end{displaymath}
     In particular, the element $\sum_{k=0}^\infty f_k\in\Rtx^s$ exists.
   \end{lemma}

   \begin{remark}[polynomial input]\label{rem:polynomialHDDwR}
     In case $m=0$, i.e. $\Rtx^s = R[x]^s$, all $f,g_1,\ldots,g_k\in
     R[x]^s$ are homogeneous and so is any polynomial appearing in our
     algorithm. Moreover, all $f_\nu$, unless $f_\nu=0$, have the same
     $x$-degree as $f$. And since there are only finitely many
     monomials of a given degree, there cannot exist an infinite
     sequence of decreasing leading monomials
     \begin{displaymath}
       \lm_>(f_0)>\lm_>(f_1)>\lm_>(f_2)>\ldots\, ,
     \end{displaymath}
     and Algorithm \ref{alg:HDDwR} has to terminate.
   \end{remark}

   \begin{remark}[weighted homogeneous input]\label{rem:HDDwRweightedHomogeneous}
     Similar to how the output is $x$-homogeneous because the input is
     $x$-homogeneous, note that if the input is weighted homogeneous
     with respect to a certain weight vector
     $w\in\RR_{<0}^m\times\RR^n$, then so is the output. This will be
     essential when computing tropical varieties over the $p$-adic
     numbers.
%     for the proof of Lemma \ref{lem:tropicalWitness}.
   \end{remark}

   \begin{example}\label{ex:HDDwR}
     Over a ground field, as in the proof of Theorem 2.1 in
     \cite{Markwig08}, all the terms of $f_\nu$ can be simultaneously
     checked for containment in $\langle \lt_>(g_1),\ldots,
     \lt_>(g_k)\rangle$, eliminating the terms which lie in the ideal
     using $g_1,\ldots,g_k$ and discarding the terms which are outside
     the ideal to the remainder. However, this is not possible if $R$
     is no field.

     Let $f=2x, g=2x+2tx+t^2x+3t^3x\in\ZZ\llbracket t\rrbracket [x]$
     and consider a weighted ordering $>=>_w$ with weight vector
     $w=(-1,1)\in\RR_{<0}\times\RR$. Then Figure
     \ref{fig:divisionCarelesslyDiscardingRemainder} illustrates a
     division algorithm, which discards any term of $f_\nu$ not
     divisible by $\lt_>(g)$ directly to the remainder. The underlined
     term marks the respective leading term.

     \begin{figure}[h]
       \centering
       \begin{tikzpicture}
         \matrix (m) [matrix of math nodes, row sep=1em, column sep=0em,
         column 3/.style={anchor=base west}]
         { f_0&=&\underline{2x} & \qquad &r\\
           f_1&=&-\underline{2tx}-\overbrace{t^2x-3t^3x}^{\text{to remainder}} \\
           f_2&=&\underline{2t^2x}+\overbrace{t^3x+3t^4x}^{\text{to remainder}} \\
           f_3&=&-\underline{2t^3x}-\overbrace{t^4x-3t^5x}^{\text{to remainder}} \\
         };
         \draw[->,shorten >=3pt,shorten <=3pt,font=\scriptsize] (m-1-2) -- node[left] {$-g$} (m-2-2);
         \draw[->,shorten >=3pt,shorten <=3pt,font=\scriptsize] (m-2-2) -- node[left] {$+t g$} (m-3-2);
         \draw[->,shorten >=3pt,shorten <=3pt,font=\scriptsize] (m-3-2) -- node[left] {$-t^2 g$} (m-4-2);
         \node[yshift=-0.5cm] at (m-4-2) {$\vdots$};

         \node[yshift=-0.77cm] (r11) at (m-1-5) {$-t^2x$};
         \node[yshift=-0.77cm] (r12) at (r11) {$3t^3x$};
         \node[yshift=-0.86cm] (r21) at (r12) {$t^3x$};
         \node[yshift=-0.77cm] (r22) at (r21) {$3t^4x$};
         \node[yshift=-0.86cm] (r31) at (r22) {$t^4x$};
         \node[yshift=-0.77cm] (r32) at (r31) {$3t^5x$};
         \node[yshift=-0.55cm] at (r32) {$\vdots$};

         \draw[draw opacity=0] (m-1-5) -- node[sloped] {$=$} (r11);
         \draw[draw opacity=0] (r11) -- node[sloped] {$-$} (r12);
         \draw[draw opacity=0] (r12) -- node[sloped] {$+$} (r21);
         \draw[draw opacity=0] (r21) -- node[sloped] {$+$} (r22);
         \draw[draw opacity=0] (r22) -- node[sloped] {$-$} (r31);
         \draw[draw opacity=0] (r31) -- node[sloped] {$-$} (r32);
       \end{tikzpicture}
       \caption{division slice by slice}
       \label{fig:divisionCarelesslyDiscardingRemainder}
     \end{figure}

     Not only would this process continue indefinitely, every term in our remainder but the first would actually be divisible by $\lt_>(g)$:
     \begin{displaymath}
       r=-t^2x-3t^3x+t^3x+3t^4x-t^4x-\ldots = -xt^2-2xt^3+2xt^4-2xt^5+\ldots \,.
     \end{displaymath}

     As we see, it is important to know when terms can be safely
     discarded to the remainder, and the only way to guarantee that is
     by proceeding term by term instead of slice by slice. And in
     order to guarantee that our result converges in the $\langle
     t\rangle$-adic topology, the order needs to be compatible with a
     weighted monomial order $>_w$ with
     $w\in\RR_{<0}^m\times\RR^{n+s}$. Figure
     \ref{fig:divisionTermByTerm} shows the same example in our
     algorithm.

     \begin{figure}[h]
       \centering
       \begin{tikzpicture}
         \matrix (m) [matrix of math nodes, row sep=1em, column sep=0em, %text height=1.5ex, text depth=0.25ex
         column 1/.style={anchor=base east},
         column 3/.style={anchor=base west}]
         { f_0&=&\underline{2x} & \qquad & r \\
           f_1&=&-\underline{2tx}-\overbrace{t^2x-3t^3x}^{\text{to be processed}} \\
           \phantom{\overbrace{t^1}^{\text{tpd}}}f_2&=&\underline{t^2x}-2t^3x+3t^4x\phantom{\overbrace{t^1}^{\text{tpd}}} & & t^2x \\
           \phantom{\overbrace{t^1}^{\text{tpd}}}f_3&=&-\underline{2t^3x}+3t^4x\phantom{\overbrace{t^1}^{\text{tpd}}} \\
           \phantom{\overbrace{t^1}^{\text{tpd}}}f_4&=&\underline{5t^4x}+t^5x+3t^6x\phantom{\overbrace{t^1}^{\text{tpd}}} & & 5t^4x \\
           \phantom{\overbrace{t^1}^{\text{tpd}}} &\vdots& \phantom{\overbrace{t^1}^{\text{tpd}}} & & t^5x \\
           \phantom{\overbrace{t^1}^{\text{tpd}}}f_7&=&0\phantom{\overbrace{t^1}^{\text{tpd}}} \\
         };
         \draw[->,shorten >=3pt,shorten <=3pt,font=\scriptsize] (m-1-2) -- node[left] {$-g$} (m-2-2);
         \draw[->,shorten >=3pt,shorten <=3pt,font=\scriptsize] (m-2-2) -- node[left] {$+tg$} (m-3-2);
         \draw[->,shorten >=3pt,shorten <=3pt,font=\scriptsize] (m-3-2) -- node[left] {to remainder} (m-4-2);
         \draw[->,shorten >=3pt,shorten <=3pt,font=\scriptsize] (m-4-2) -- node[left] {$+t^2g$} (m-5-2);
         \draw[->,shorten >=-3pt,shorten <=3pt,font=\scriptsize] (m-5-2) -- node[left] {to remainder} (m-6-2);
         \draw[->,shorten >=3pt,shorten <=3pt,font=\scriptsize] (m-6-2) -- node[left] {to remainder} (m-7-2);

         \node[yshift=-0.77cm] (rEnd) at (m-6-5) {$3t^6x$};

         \draw[draw opacity=0] (m-1-5) -- node[sloped] {$=$} ($(m-1-5)+(0,-1)$);
         \node[anchor=south] at (m-5-5.north) {$+$};
         \node[anchor=south] at (m-6-5.north) {$+$};
         \draw[draw opacity=0] (rEnd) -- node {$+$} (m-6-5);
       \end{tikzpicture}
       \caption{division term by term}
       \label{fig:divisionTermByTerm}
     \end{figure}

     We obtain a representation satisfying (DD1), (DD2) and (DDH):
     \begin{displaymath}
       f = (\underbrace{1-t-t^3}_{=q})\cdot g + (\underbrace{xt^2+5xt^4+xt^5+3xt^6}_{=r}).
     \end{displaymath}
   \end{example}

   Having constructed a homogeneous determinate division with
   remainder, we will now introduce homogenization, dehomogenization
   and the ecart to continue with a weak division with remainder.

   \begin{definition}[Homogenization and dehomogenization]
     For an element $f=\sum_{\beta,\alpha,i}c_{\alpha,\beta,i}\cdot
     t^\beta x^\alpha\cdot e_i\in\Rtx^s$ we define its
     \emph{homogenization} to be
     \begin{displaymath}
       f^h:=\sum_{\alpha,\beta,i} c_{\alpha,\beta,i}\cdot t^\beta x_0^{\deg_x(f)-|\alpha|} x^\alpha\cdot e_i\in R\llbracket t \rrbracket[x_h]^s
     \end{displaymath}
     with $x_h=(x_0,x)=(x_0,x_1,\ldots,x_n)$.
     And for an element $F\in R\llbracket t\rrbracket[x_h]^s$ we define its \emph{dehomogenization} to be $F|_{x_0=1}\in\Rtx^s$.
   \end{definition}

   \begin{remark}[Homogenization and dehomogenization]\label{re:homogenization}
     Any monomial ordering $>$ on $\Mon^s(t,x)$, can be naturally extended to an ordering $>_h$ on $\Mon^s(t,x_0,x)$ through
     \begin{align*}
       a >_h b \quad :\Longleftrightarrow \quad & \deg_{x_h}(a)>\deg_{x_h}(b) \text{ or } \\
       & \deg_{x_h}(a)=\deg_{x_h}(b) \text{ and } a|_{x_0=1} > b|_{x_0=1}.
     \end{align*}
     Defining the \emph{ecart} of an element $f\in\Rtx^s$ with respect to $>$ to be
     \begin{displaymath}
       \ecart_>(f) := \deg_x(f)-\deg_x(\lm_>(f))\in\NN,
     \end{displaymath}
     one can show that for any elements $g,f\in\Rtx^s$ and any $x_h$-homogeneous $F\in R\llbracket t\rrbracket[x_h]$:
     \begin{enumerate}
     \item $f=(f^h)^d$,
     \item $F=x_0^{\deg_{x_h}(F)-\deg_x(F^d)}\cdot (F^d)^h$,
     \item $\lt_{>_h}(f^h) = x_0^{\ecart_>(f)}\cdot\lt_>(f)$,
     \item $\lt_{>_h}(F) = x_0^{\ecart_>(F^d)+\deg_{x_h}(F)-\deg_x(F^d)}\cdot \lt_>(F^d)$,
     \item $\lm_{>_h}(g^h) | \lm_{>_h}(f^h) \quad \Longleftrightarrow$\phantom{$x_0^{\ecart_>(F^d)+\deg_{x_h}(F)-\deg_x(F^d)}$} \\ \phantom{$\quad$} $\lm_>(g) | \lm_>(f)$ and $\ecart_>(g)\leq\ecart_>(f)$,
     \item $\lm_{>_h}(g^h) \mid \lm_{>_h}(F) \quad \Longleftarrow$\phantom{$x_0^{\ecart_>(F^d)+\deg_{x_h}(F)-\deg_x(F^d)}$} \\ \phantom{$\quad$} $\lm_>(g) | \lm_>(F^d)$ and $\ecart_>(g)\leq\ecart_>(F^d)$.
     \end{enumerate}
   \end{remark}

   With this preparation we are now able to formulate and prove weak
   division with remainder.

   \begin{algorithm}[DwR, weak division with remainder] \label{alg:DwR}\skipalgorithm
     \begin{algorithmic}[1]
       \REQUIRE{$(f,G,>)$, where $f\in\Rtx^s$ and $G=(g_1,\ldots,g_k)$
         is a $k$-tuple in $\Rtx^s$ and $>$ a weighted $t$-local
         monomial ordering on $\Mon^s(t,x)$.}
       \ENSURE{$(u,Q,r)$, where $u\in \Rtx$ with $\lt_>(u)=1$,
         $Q=(q_1,\ldots,q_k)\subseteq\Rtx^k$
         and $r\in \Rtx^s$ such that
         \begin{displaymath}
           u\cdot f=q_1\cdot g_1+\hdots+q_k\cdot g_k+r
         \end{displaymath}
         satisfies \kurz{(ID1) and (ID2).

         }
         \lang{\begin{description}
         \item[\rm (ID1)] $\lm_>(f)\geq\lm_>(q_i\cdot g_i)$ for $i=1,\ldots,k$ and
         \item[\rm (ID2)] $\lt_>(r)\notin \langle\lt_>(g_1),\ldots,\lt_>(g_k)\rangle$, unless $r=0$.
         \end{description}}
         Moreover, the algorithm requires only a finite number of recursions.}
       \IF{$f\neq 0$ and $\lt_>(f)\in\langle\lt_>(g_1),\ldots,\lt_>(g_k)\rangle$}
       \STATE Set $D:=\{ g_i\in G \mid \lm_>(g_i) \text{ divides } \lm_>(f) \}$ and $D' := \emptyset$.
       \WHILE{$\lt_>(f)\notin \langle \lt_>(g_i)\mid g_i\in D' \rangle$}
       \STATE Pick $g\in D$ with minimal ecart. %$\ecart_>(g)=\min\{\ecart_>(g')\mid g'\in D\}$.
       \STATE Set $D' := D' \cup \{g\}$ and $D:=D\setminus \{g\}$.
       \ENDWHILE
       \IF{$e:= \max\{\ecart_>(g)\mid g\in D' \} -  \ecart_>(f)>0$}
       \STATE Compute
       \begin{displaymath}
         ((Q_1',\ldots,Q_k'),R'):=\HDDwR(x_0^e\cdot f^h,(\lt_>(g_1^h),\ldots,\lt_>(g_k^h)),>_h).
       \end{displaymath}
       \STATE Set $f':=(x_0^e\cdot f^h - \sum_{i=1}^k Q_i'\cdot g_i^h)^d$.
       \STATE Run \begin{displaymath}
         (u'',(q_1'',\ldots,q_{k+1}''),r):=\DwR(f',(g_1,\ldots,g_k,f),>).
       \end{displaymath}
       \STATE Set $q_i:= q_i''+u''\cdot Q_i'^d$, $i=1,\ldots,k$.
       \STATE Set $u:=u''-q_{k+1}''$.
       \ELSE
       \STATE Compute \begin{displaymath}
         ((Q_1',\ldots,Q_k'),R'):=\HDDwR(f^h,(g_1^h,\ldots,g_k^h),>_h).
       \end{displaymath}
       \STATE Run \begin{displaymath}
         (u,(q_1'',\ldots,q_k''),r):=\DwR((R')^d,(g_1,\ldots,g_k),>).
       \end{displaymath}
       \STATE Set $q_i:=q_i''+u\cdot Q_i'^d$, $i=1,\ldots,k$.
       \ENDIF
       \ELSE
       \STATE Set $(u,(q_1,\ldots,q_k),r):=(1,(0,\ldots,0),f)$.
       \ENDIF
       \RETURN{$(u,(q_1,\ldots,q_k),r)$.}
     \end{algorithmic}
   \end{algorithm}
   \begin{proof}
     \emph{Finiteness of recursions:} For sake of clarity, label all
     the objects appearing in the $\nu$-th recursion
     step by a subscript $\nu$. For example the ecart $e_\nu\in\NN$,
     the element $f_\nu\in\Rtx^s$ and the subset
     $G_\nu\subseteq\Rtx^s$.

     Since $G_1^h\subseteq G_2^h\subseteq G_3^h\subseteq \hdots$, we
     have an ascending chain of leading ideals in $\Rtxx^s$, which
     eventually stabilizes unless the algorithm terminates beforehand
     \begin{displaymath}
       \LT_{>_h}(G_1^h)\subseteq  \LT_{>_h}(G_2^h)\subseteq \ldots \subseteq \LT_{>_h}(G_N^h) = \LT_{>_h}(G_{N+1}^h) = \ldots.
     \end{displaymath}

     Assume $e_N>0$. Then we'd have $f_N\in G_{N+1}$, and thus
     \begin{displaymath}
       \lt_{>_h}(f_N^h)\in \lt_{>_h}(G_{N+1}^h)=\lt_{>_h}(G_N^h).
     \end{displaymath}
     To put it differently, we'd have
     \begin{displaymath}
       \lt_{>_h}(f_N^h)\in\langle \lt_{>_h}(g^h)\mid g^h\in G_N^h \text{ with } \lm_{>_h}(g^h) \text{ divides } \lm_{>_h}(f_N^h) \rangle,
     \end{displaymath}
     which by Remark \ref{re:homogenization} (5) would imply that
     \begin{align*}
       \lt_>(f_N) \in \langle \lt_>(g)\mid g\in G_N \text{ with }& \lm_>(g) \text{ divides } \lm_>(f_N), \\
       & \text{ and }\ecart_>(g)\leq\ecart_>(f_N) \rangle.
     \end{align*}
     Consequently, we'd get
     \begin{displaymath}
       D'_N\subseteq \{ g\in G_N\mid \lm_>(g) \text{ divides } \lm_>(f_N) \text{ and } \ecart_>(g)\leq\ecart_>(f_N)\},
     \end{displaymath}
     contradicting our assumption
     \begin{displaymath}
       e_N = \max\{ \ecart_>(g) \mid g \in D_N' \} - \ecart_>(f_N) \overset{!}{>} 0. \quad \lightning
     \end{displaymath}
     Therefore we have $e_N\leq 0$. By induction we conclude that
     $e_\nu\leq 0$ for all $\nu\geq N$, i.e. that we will exclusively
     run through steps 14-16 of the ``else'' case from the $N$-th
     recursion step onwards.

     By the properties of HDDwR we know that in particular
     \begin{displaymath}
       \lt_{>_h}(R_N')\notin \LT_{>_h}(G_N^h).
     \end{displaymath}
     Now assume that the recursions would not stop with the next recursion.
     That means there exists a $D'_{N+1} \subseteq D_{N+1}$ with
     \begin{displaymath}
       \lt_>((R_N')^d)=\lt_>(f_{N+1})\in\langle \lt_>(g) \mid g\in D'_{N+1} \rangle
     \end{displaymath}
     such that
     $e_{N+1}=\max\{\ecart_>(g)\mid g\in D'_{N+1} \} - \ecart_>((R_N')^d)\leq 0$.
     From Remark \ref{re:homogenization} (6), this immediately implies the following contradiction
     \begin{displaymath}
       \lt_{>_h}(R_N')\in \LT_{>_h}(G_{N+1}^h)=\LT_{>_h}(G_N^h). \quad \lightning
     \end{displaymath}
     Hence the algorithm terminates after the $N+1$-th recursion step.

     \emph{Correctness:} We make an induction on the number of recursions, say $N\in\NN$.
     If $N=1$ then either $f=0$ or $\lt_>(f)\notin \langle\lt_>(g_1),\ldots,\lt_>(g_k)\rangle$, and in both cases
     \begin{displaymath}
       1\cdot f=0\cdot g_1+\hdots+0\cdot g_k+f
     \end{displaymath}
     satisfies (ID1) and (ID2).

     So suppose $N>1$ and consider the first recursion step.
     If $e \leq 0$, then by the properties of HDDwR the representation
     \begin{displaymath}
       f^h=Q_1'\cdot g_1^h+\hdots+ Q_k'\cdot g_k^h+R'
     \end{displaymath}
     satisfies (DD1), (DD2) and (DDH). (DD1) and (DD2) imply (ID1), which means that for each $i=1,\hdots,k$ we have
     \begin{flushleft}
       $x_0^{\ecart_>(f)}\cdot\lm_>(f)=\lm_{>_h}(f^h)\overset{(\text{ID1})}{\geq_h} \lm_{>_h}(Q_i')\cdot\lm_{>_h}(g_i^h)=\ldots$
     \end{flushleft}
     \begin{flushright}
       $\ldots=x_0^{a_i+\ecart_>(g_i)}\cdot\lm_>(Q_i'^d)\cdot\lm_>(g_i)$
     \end{flushright}
     for some $a_i\geq 0$. Since $f^h$ and $Q_i'\cdot g_i^h$ are both $x_h$-homogeneous of the same $x_h$-degree
     by (DDH), the definition of the homogenized ordering $>_h$ implies
     \begin{equation}
       \label{eq:dwr11}
       \lm_>(f)\geq\lm_>(Q_i'^d)\cdot\lm_>(g_i) \text{ for all } i=1,\ldots,k.
     \end{equation}
     Moreover, by induction the representation $u\cdot R'^d=q_1''\cdot g_1+\ldots+q_k''\cdot g_k+r$
     satisfies (ID1), (ID2) and $\lt_>(u)=1$, the first implying that
     \begin{equation}
       \label{eq:dwr12}
       \lm_>(f)\overset{(\ref{eq:dwr11})}{\geq}\lm_>\underbrace{\left(f-\sum_{i=1}^kQ_i'^d\cdot g_i\right)}_{=R'^d}\overset{\text{(ID1)}}{\geq}\lm_>(q_i''\cdot g_i).
     \end{equation}
     Therefore, the representation
     \begin{displaymath}
       u\cdot f=\sum_{i=1}^k(q_i''+u\cdot Q_i'^d)\cdot g_i+r
     \end{displaymath}
     satisfies (ID1) by (\ref{eq:dwr11}), (\ref{eq:dwr12}), $\lt_>(u)=1$ and (ID2) by induction.

     Similarly, if $e>0$, then by the properties of HDDwR the representation
     \begin{displaymath}
       x_0^e\cdot f^h=Q_1'\cdot\lt_{>_h}(g_1^h)+\ldots+Q_k'\cdot\lt_{>_h}(g_k^h)+R'
     \end{displaymath}
     satisfies (DD1), (DD2) and (DDH). (DD1) and (DD2) imply (ID1), which means that for each $i=1,\ldots,k$ we have
     \begin{flushleft}
       $x_0^{e+\ecart_>(f)}\cdot\lm_>(f)=\lm_{>_h}(x_0^e\cdot f^h) \geq \ldots$
     \end{flushleft}
     \begin{flushright}
       $\ldots\geq\lm_{>_h}(Q_i')\cdot\lm_{>_h}(\lt_{>_h}(g_i^h))=x_0^{a_i+\ecart_>(g_i)}\cdot\lm_>(Q_i'^d)\cdot\lm_>(g_i),$
     \end{flushright}
     for some $a_i\geq 0$.
     Since $x_0^e\cdot f^h$ and $Q_i'\cdot \lt_{>_h}(g_i^h)$ are both
     $x_h$-homogeneous of the same $x_h$-degree by (DDH), the
     definition of the homogenized ordering $>_h$ implies
     \begin{equation}
       \label{eq:dwr21}
       \lm_>(f)\geq\lm_>(Q_i'^d)\cdot\lm_>(g_i).
     \end{equation}
     Moreover, by induction the representation
     $u''\cdot f'=\sum_{i=1}^kq_i''\cdot g_i+q_{k+1}''\cdot f+r $
     satisfies (ID1), (ID2) and $\lt_>(u'')=1$ with the first implying that
     \begin{equation}
       \label{eq:dwr22}
       \lm_>(f)\overset{(\ref{eq:dwr21})}{\geq}\underbrace{\lm_>\left(f-\sum_{i=1}^kQ_i'^d\cdot g_i\right)}_{=\lm_>(R'^d)}\overset{\text{(ID1)}}{\geq} \lm_>(q_i''\cdot g_i).
     \end{equation}
     Therefore, the representation
     \begin{displaymath}
       u\cdot f=\sum_{i=1}^k(q_i''+u''\cdot Q_i'^d)\cdot g_i+r, \text{ with } u=u''-q_{k+1}''
     \end{displaymath}
     satisfies (ID1) by (\ref{eq:dwr21}), (\ref{eq:dwr22}), $\lt_>(u'')=1$ and (ID2) by induction.

     To see that $\lt_>(u)=1$, observe that \begin{displaymath}
       \lt_{>_h}(x_0^e\cdot f^h)\in \langle\lt_>(g_1^h),\ldots,\lt_>(g_k^h)\rangle,
     \end{displaymath}
     which is why
     \begin{displaymath}
       \lm_>(f)=\lm_{>_h}(x_0^e \cdot f^h)^d>\lm_{>_h}\left(x_0^e\cdot f^h - \sum_{i=1}^k Q_i'\cdot g_i^h\right)^d=\lm_>(f').
     \end{displaymath}
     Thus $\lm_>(f)>\lm_>(f')\geq\lm_>(q_{k+1}'')\cdot\lm_>(f)$, which necessarily implies $\lm(q_{k+1}'')<1$.
     By induction we get $\lt_>(u)=\lt_>(u'')=1$.
   \end{proof}

   \begin{remark}[polynomial input]\label{rem:polynomialDwR}
     If the input is polynomial, $f,g_1,\ldots,g_k\in R[t,x]^s$, then
     we can regard them as elements of $R\llbracket t' \rrbracket
     [x']=R[t,x]$ with $t'=()$ and $x'=(t,x)$.
     In that case, our homogeneous determinate divisions with
     remainder terminates by Remark \ref{rem:polynomialHDDwR}, and
     hence so does our weak division with remainder. In particular,
     the output $q_1,\ldots,q_k,r$ will be polynomial as well.
   \end{remark}

   The next corollary will prove to be very useful in Theorem
   \ref{thm:standardBasis}, though not for elements in $\Rtx^s$, but
   for elements in $\Rtx^k$ under the Schreyer ordering.

   \begin{corollary}\label{cor:sid2prep}
     Let $>$ be a $t$-local monomial ordering and $g_1,\ldots,g_k\in \Rtx^s$.
     Then any $f \in \Rtx^s$ has a weak division with remainder
     \begin{displaymath}
       u\cdot f=q_1\cdot g_1+\ldots+q_k\cdot g_k+r
     \end{displaymath}
     with $r=\sum_{j=1}^s r_je_j\in\Rtx^s$ satisfying \kurz{(SID2).}
     \lang{\begin{description}
     \item[\rm (SID2)] $\lt_>(r_j\cdot e_j)\notin\langle \lt_>(g_1),\ldots,\lt_>(g_k) \rangle$, unless $r_j=0$, for $j=1,\ldots,s$.
     \end{description}}
   \end{corollary}
   \begin{proof}
     We make an induction on $s$, in which the base case $s=1$ follows from Algorithm \ref{alg:DwR},
     as condition (SID2) coincides with (ID2).

     Suppose $s>1$. By Algorithm \ref{alg:DwR} there exists a weak division with remainder
     \begin{equation}\label{eq:sid1}
       u\cdot f = q_i\cdot g_1 + \ldots + q_k\cdot g_k + r.
     \end{equation}
     If $r=0$, then the representation satisfies (SID2) and we're
     done. If $r\neq 0$, there is a unique $j\in\{1,\ldots,s\}$ such
     that $\lt_>(r)\in\Rtx\cdot e_j$.
     For sake of simplicity, suppose that $j=s$ and that $g_1,\ldots,g_k$ are ordered in such that
     \begin{displaymath}
       \underbrace{\lt_>(g_1),\ldots,\lt_>(g_l)}_{\notin\Rtx\cdot e_s}, \quad \underbrace{\lt_>(g_{l+1}),\ldots,\lt_>(g_s)}_{\in\Rtx\cdot e_s} \quad \text{ for some } 1\leq l < s.
     \end{displaymath}

     Consider the projection
     \begin{displaymath}
       \sigma: \Rtx^s  \longrightarrow \Rtx^{s-1}, \quad
       (p_1,\ldots,p_s) \longmapsto     (p_1,\ldots, p_{s-1}),
     \end{displaymath}
     the inclusion
     \begin{displaymath}
       \iota: \Rtx^{s-1}   \longrightarrow \Rtx^s, \quad
       (p_1,\ldots, p_{s-1}) \longmapsto     (p_1,\ldots, p_{s-1},0),
     \end{displaymath}
     and let $>_\ast$ denote the restriction of $>$ on $\Mon(t,x)^{s-1}$.
     Note that we have
     \begin{enumerate}[leftmargin=*]
     \item for $h\in \Rtx^{s-1}$: $\lm_>(\iota(h))=\iota(\lm_{>_\ast}(h))$,
     \item for $i=1,\ldots,l$: $\lm_>(g_i)=\lm_>(\iota(\sigma(g_i)))$.
     \end{enumerate}

     By induction, there exists a weak division with remainder of $\sigma(r)\in\Rtx^{s-1}$ satisfying (SID2), say
     \begin{equation}\label{eq:sid2}
       u'\cdot\sigma(r) = q_1'\cdot \sigma(g_1) + \ldots + q_l'\cdot \sigma(g_l) + r'.
     \end{equation}
     Writing $r=\sum_{j=1}^s r_j\cdot e_j$ and $r'=\sum_{j=1}^{s-1}
     r_j'\cdot e_j$, we want to show that the following constructed
     representation
     \begin{displaymath}
       u\cdot u'\cdot f = \sum_{i=1}^l (u'\cdot q_i + q_i') \cdot g_i + \sum_{i=l+1}^k u'\cdot q_i \cdot g_i + r'' \text{ with } r''=\sum_{j=1}^{s-1}r_j'\cdot e_j + r_s\cdot e_s
     \end{displaymath}
     is a weak division with remainder satisfying (SID2). \\

     As (\ref{eq:sid1}) satisfies (ID2), (\ref{eq:sid2}) satisfies (ID1), and $\lt_>(r)\in \Rtx_>\cdot e_s$, we obtain for $i=1,\ldots,l$
     \begin{flushleft}
       $\lm_>(f)\geq\lm_>(r) > \lm_>(\iota (\sigma(r))) \geq \lm_>(\iota (q_i'\cdot \sigma(g_i))) = \ldots$
     \end{flushleft}
     \begin{flushright}
       $\ldots = \lm_>(q_i'\cdot \iota (\sigma(g_i))) = \lm_>(q_i'\cdot g_i).$
     \end{flushright}
     Now since (\ref{eq:sid1}) satisfies (ID1) and $\lt_>(u)=1=\lt_>(u')$, we have for $i\leq l$
     \begin{displaymath}
       \lm_>(u\cdot u'\cdot f) =\lm_>(f) \geq \lm_>((u'\cdot q_i+q_i')\cdot g_i)
     \end{displaymath}
     and for $i>l$
     \begin{displaymath}
       \lm_>(u\cdot u'\cdot f) =\lm_>(f) \geq \lm_>(q_i\cdot g_i) = \lm_>(u'\cdot q_i\cdot g_i),
     \end{displaymath}
     proving that our constructed representation satisfies (ID1).

     Moreover, (SID2) of (\ref{eq:sid2}) tells us that for $j=1,\ldots,s-1$
     \begin{displaymath}
       \lt_{>_\ast}(r_j'\cdot e_j) \notin \langle \lt_{>_\ast}(\sigma(g_1)),\ldots,\lt_{>_\ast}(\sigma(g_l))\rangle, \text{ unless } r_j'=0,
     \end{displaymath}
     And because $\lt_>(g_i) \in \Rtx\cdot e_s$ for $i>l$, we get for $j=1,\ldots,s-1$
     \begin{displaymath}
       \lt_>(r_j'\cdot e_j) \notin \langle \lt_>(g_1),\ldots,\lt_>(g_s)\rangle, \text{ unless } r_j'=0.
     \end{displaymath}
     In addition, by (ID2) of (\ref{eq:sid1}), we have
     \begin{displaymath}
       \lt_>(r_s'\cdot e_s) = \lm_>(r) \notin \langle\lt_>(g_1),\ldots,\lt_>(g_s)\rangle,
     \end{displaymath}
     which completes the proof that our constructed representation satisfies (SID2). By Proposition \ref{prop:dwrs} this implies (ID2).
   \end{proof}

   % There is a problem in only having a weak division with remainder on $\Rtx$,
   % which can be avoided by considering a bigger ring. It will turn out that this workaround has no
   % effect on our applications.

   We will now introduce localizations at monomial orderings. More
   than just a convenience to get rid of the $u$ with $\lm_>(u)=1$ in
   our weak division with remainder, localization at monomial
   orderings allows geometers to compute in localizations at ideals
   generated by variables. It is a technique that has been applied in
   the study of isolated singularities to great success.

   \begin{definition}[Localization at monomial orderings]
     For a $t$-local monomial ordering $>$ on $\Mon(t,x)$, we define
     \begin{displaymath}
       S_>:=\{u\in \Rtx\mid \lt_>(u) = 1\} \text{ and } \Rtx_>:=S_>^{-1}\Rtx.
     \end{displaymath}
     We will refer to $\Rtx_>$ as $\Rtx$ \emph{localized at the monomial ordering $>$}.

     Let $>$ be a module monomial ordering on $\Mon^s(t,x)$. Recall
     that it restricts to the same monomial ordering on $\Mon(t,x)$ in
     each component by Definition~\ref{def:basicsModule}, which we
     will denote by $>_\Rtx$. We then define for any $k\in\NN$
     \begin{displaymath}
       \Rtx_>^s:=S_{>_\Rtx}^{-1}\left(\Rtx^s\right).
     \end{displaymath}
     We will refer to $\Rtx_>^s$ as $\Rtx^s$ \emph{localized at the
       monomial ordering $>$}. For $s=1$, it coincides with the first
     definition.

     Our definitions on $\Rtx^s$ extend naturally to $\Rtx_>^s$, since
     for any element $f\in\Rtx_>^s$ there exists an element $u\in S_>$
     such that $u\cdot f\in\Rtx^s$. We define the \emph{leading
       monomial, leading coefficient} and \emph{leading term} of $f$
     with respect to $>$ to be that of $u\cdot f\in\Rtx^s$. The
     \emph{leading module} of a submodule $M\leq\Rtx_>^s$ is again the
     module generated by the leading terms of its elements.

     And given $f,g_1,\ldots,g_k,r=\sum_{j=1}^s r_j\cdot e_j \in \Rtx_>^s$, we say a representation
     \begin{displaymath}
       f=q_1\cdot g_1+\ldots+q_k\cdot g_k+r
     \end{displaymath}
     satisfies
     \begin{description}
     \item[\rm (ID1)] if $\lm_>(f)\geq\lm_>(q_i\cdot g_i)$ for all $i=1,\ldots,k$,
     \item[\rm (ID2)] if $\lt_>(r)\notin \langle \lt_>(g_1),\ldots,\lt_>(g_k)\rangle_\Rtx$, unless $r=0$,
     \item[\rm (DD1)] if no term of $q_i\cdot \lt_>(g_i)$ lies in $\langle \lt_>(g_j) \mid j<i \rangle_\Rtx$ for all $i=1,\ldots,k$,
     \item[\rm (DD2)] if no term of $r$ lies in $\langle \lt_>(g_1),\ldots,\lt_>(g_k) \rangle$,
     \item[\rm (SID2)] if $\lt_>(r_j\cdot e_j)$ does not lie in $\langle \lt_>(g_1),\ldots,\lt_>(g_k) \rangle_\Rtx$, unless $r_j=0$, for $j=1,\ldots,s$.
     \end{description}
     We will refer to a representation satisfying (ID1) and (ID2) as
     \emph{(indeterminate) division with remainder}, and we will refer
     to a representation satisfying (DD1) and (DD2) as
     \emph{determinate division with remainder}. In each of these two cases
     we call $r$ a \emph{remainder} or \emph{normal form} of $f$ with
     respect to $(g_1,\ldots,g_k)$. Moreover, if the remainder $r$ is
     zero, we call the representation a \emph{standard representation}
     of $f$ with respect to $(g_1,\ldots,g_k)$.
   \end{definition}

   With these notions, Corollary \ref{cor:sid2prep} then implies:

   \begin{corollary}\label{cor:sid2>}
     Let $>$ be a monomial ordering and $g_1,\ldots,g_k\in \Rtx^s_>$.
     Then any $f \in \Rtx^s_>$ has a division with remainder with respect to $g_1,\ldots,g_k$
     satisfying (SID2).
   \end{corollary}

   \section{Standard bases and syzygies}

   In this section, we introduce standard bases for rings satisfying
   Convention~\ref{con:groundRing}.
   We also incorporate some remarks on possible
   optimizations for $R$ being a principal ideal domain. Similar to
   the classical theory, it opens with introducing the Schreyer
   ordering and syzygies, and finishes with proving Buchberger's
   criterion.

   \begin{definition}\label{def:standardBases}
     Let $>$ be a $t$-local monomial ordering  on $\Mon(t,x)^s$ and
     $M\leq \Rtx^s$ or $M\leq \Rtx_>^s$.  A \emph{standard basis} of
     $M$ with respect to $>$ is a finite set
     $G\subseteq M$ with
     \begin{displaymath}
       \LT_>(G)=\LT_>(M)
     \end{displaymath}
     where  $\LT_>(G):=\langle \LT_>(g)\mid g\in G\rangle$. $G$ is
     simply called a \emph{standard basis} with respect to $>$, if $G$
     is a standard basis of $\langle G\rangle_{\Rtx_>}$ with respect
     to $>$.
   \end{definition}

   With this definition we get the usual results for standard
   bases. We will formulate them, but we will only prove them if
   the proof has to be adjusted due to the fact that the base ring is
   not a field. For the existence of standard bases it is important to
   note, that our base ring is noetherian.

   \begin{proposition}
     For any monomial ordering $>$ all submodules of $\Rtx^s$ and $\Rtx_>^s$ have a standard basis.
   \end{proposition}
   \lang{\begin{proof}
     Let $M\leq \Rtx^s$ resp.~$M\leq\Rtx_>^s$  be a submodule. Since $R$ is noetherian, so are
     $\Rtx^s$ and $\Rtx_>^s$, and $\LT_>(M)\leq \Rtx^s$ has a finite generating set
     $h_1,\ldots,h_k$.
     Because
     \begin{displaymath}
       \LT_>(M)=\langle \LT_>(g)\mid g\in M\rangle \overset{!}{=} \{ \LT_>(g)\mid g\in M \},
     \end{displaymath}
     there exist $g_1,\ldots,g_k$ with $\LT_>(g_i)=h_i$ forming a standard basis of $M$.
   \end{proof}}

   Computing weak normal forms is essential in the standard
   bases algorithm. \lang{While it can be essentially done by computing a
   division with remainder and discarding everything but the
   remainder, as in the following algorithm, the fact that everything
   but the remainder is discarded may be used for some optimization in
   the division algorithm, which we leave out for sake of clarity.}

   \begin{algorithm}[normal form]\label{alg:normalForm} \skipalgorithm
     \begin{algorithmic}[1]
       \REQUIRE{$(f,G,>)$, where $f\in\Rtx$, $G=(g_1,\ldots,g_k)$ a
         $k$-tuple in $\Rtx^s$ and $>$ a $t$-local monomial ordering.}
       \ENSURE{$r=\NF(f,G,>)\in \Rtx$, a normal form of $f$ with respect to $G$ and $>$.}
       \STATE Use Algorithm \ref{alg:DwR} to compute a division with remainder,
       \begin{displaymath}
         (u,(q_1,\ldots,q_k),r)=\DwR(f,G,>).
       \end{displaymath}
       \RETURN{$r$.}
     \end{algorithmic}
   \end{algorithm}

   \begin{remark}[polynomial input]\label{rem:polynomialNormalForm}
     Should the input be polynomial, i.e. $f\in R[t,x]$ and
     $G\subseteq R[t,x]$, then by Remark~\ref{rem:polynomialDwR} we
     automatically obtain a polynomial normal form $\NF(f,G,>) \in
     R[t,x]$.
   \end{remark}

   \begin{convention}
     For the remainder of the section, fix a $t$-local monomial ordering $>$ on $\Mon(t,x)^s$.
   \end{convention}

   \begin{proposition}\label{idealMembership}
     Let $M\leq \Rtx^s$ be a module and let $G=\{g_1,\ldots,g_k\}$ be a standard basis of $M$. Then
     given an element $f\in \Rtx$ and a weak division with remainder
     \begin{displaymath}
       u\cdot f = q_1\cdot g_1+\hdots+q_k\cdot g_k+r,
     \end{displaymath}
     we have $f \in M$ if and only if $r = 0$. In particular, we see that $M=\langle G \rangle$
   \end{proposition}
   \lang{\begin{proof}\
     If $r=0$, then obviously $f\in\langle G \rangle \subseteq J$.
     Conversely, if $f\in J$, then $r=u\cdot f-q_1\cdot g_1+\hdots+q_k\cdot g_k \in J$ and therefore
     $\lt_>(r)\in\LT_>(J)=\LT_>(G)$. Hence $r=0$ by (ID2).

     We obviously have $M \supseteq \langle G \rangle$. For the
     converse, note that $u\in\Rtx_>$ with $\lt_>(u)=1$ is a unit, and
     hence the weak division with remainder implies $M\subseteq
     \langle G \rangle$.
   \end{proof}}

   \begin{proposition}\label{prop:qbuchberger0}
     Let $M$ be a submodule of $\Rtx_>^s$ (resp.~of $\Rtx^s$)
     and let $G=\{g_1,\ldots,g_k\}\subseteq M$. Then the following statements are equivalent:
     \begin{enumerate}[label=\normalfont(\alph*), leftmargin=*]
     \item $G$ is a standard basis of $M$.
     \item Every (weak) normal form of any element in $M$ with respect to $G$ is zero.
     \item Every element in $M$ has a (weak) standard representation with respect to $G$.
     \end{enumerate}
   \end{proposition}
   \lang{\begin{proof}
     By Proposition \ref{idealMembership} (a) implies (b), and the
     implication (b) to (c) is true by Corollary \ref{cor:sid2>}. And
     if any $f\in J$ has a standard representation
     \begin{displaymath}
       f=q_1\cdot g_1+\ldots +q_k\cdot g_k,
     \end{displaymath}
     then, since $\lm_>(f)\geq\lm_>(q_i\cdot g_i)$ for $i=1,\ldots,k$,
     there can be no total cancellation of the leading terms on the
     right hand side. Hence $\lt_>(f)\in\LT_>(G)$, and (c) implies
     (a).
   \end{proof}}

%    As an immediate consequence, we get:

%    \begin{corollary}\label{cor:standardBasisCriterion}
%      Let $M$ be a submodule of $\Rtx^s$ and $G=\{g_1,\ldots,g_k\}\subseteq M$. Then the following statements are equivalent:
%      \begin{enumerate}[label=\normalfont(\alph*),leftmargin=*]
%      \item $G$ is a standard basis of $M$.
%      \item Every weak normal form of any element in $M$ with respect to $G$ is zero.
%      \item Every element in $M$ has a weak standard representation with respect to $G$.
%      \end{enumerate}
%    \end{corollary}
   % \begin{proof}
   %   Follows from proposition \ref{qbuchberger0}, as a standard basis of
   %   $J:=\langle I \rangle_{\Rtx_>}$ yields a standard basis of $I$,
   %   and as any weak standard representation in $\Rtx$ yields a regular standard representation
   %   in $\Rtx_>$, and vice versa.
   % \end{proof}

   Also note that this in particular implies for $x$-homogeneous modules that being a standard basis only depends on the leading monomials.

   \begin{corollary}\label{cor:standardBasisMultipleOrderings}
     Let $G$ be an $x$-homogeneous standard basis of an
     $x$-homogeneous module $M\leq \Rtx$ with respect to $>$. Let $>'$
     be another $t$-local monomial ordering on $\Mon^s(t,x)$ such that
     $\lm_{>'}(g)=\lm_{>}(g)$ for all $g\in G$.
     Then $G$ is also a standard basis of $M$ with respect to $>'$.
   \end{corollary}
   \begin{proof}
     By Algorithm \ref{alg:HDDwR}, for any $f\in M=\langle G\rangle$
     we can compute a determinate division with remainder $0$ with
     respect to $>$,
     \begin{displaymath}
       f=q_1\cdot g_1+\ldots +q_k\cdot g_k+0.
     \end{displaymath}
     However, since the conditions (DD1) and (DD2) are only dependant
     on $\lm_>(g_i)=\lm_{>'}(g_i)$, this is also a valid determinate
     division with remainder under $>'$. By
     Proposition~\ref{prop:dwrs}, this is in particular a valid
     division with remainder, proving that $G$ is also a standard
     basis with respect to $>'$.
   \end{proof}

   \begin{definition}[Syzygies and Schreyer ordering]\label{re:syz}
     Given a $k$-tuple $G=(g_1,\ldots,g_k)$ in  $\Rtx_>^s$, we define the \emph{Schreyer ordering} $>_S$ on
     $\Mon^k(t,x)$ associated to $G$ and $>$ to be
     \begin{align*}
       &t^\alpha\cdot x^\beta\cdot \varepsilon_i >_S  t^{\alpha'}\cdot x^{\beta'}\cdot \varepsilon_j
       \quad :\Longleftrightarrow \\
       & \qquad  t^\alpha\cdot x^\beta \cdot \lm_>(g_i) >  t^{\alpha'}\cdot x^{\beta'} \cdot \lm_>(g_j) \text{ or } \\
       & \qquad  t^\alpha\cdot x^\beta \cdot \lm_>(g_i) =  t^{\alpha'}\cdot x^{\beta'} \cdot \lm_>(g_j) \text{ and } i>j.
     \end{align*}
     Note that we distinguish between the canonical basis elements
     $e_j$ of the free module $\Rtx_>^s$ and the canonical basis
     elements $\varepsilon_i$ of the free module $\Rtx_>^k$.

     Moreover, observe that $>_S$ and $>$ restrict to the same monomial ordering on $\Mon(t,x)$, so that
     \begin{displaymath}
       \Rtx^k_{>_S}=S_{>_{S,\Rtx}}^{-1}\Rtx^k = S_{>_{\Rtx}}^{-1} \Rtx^k = \Rtx^k_>.
     \end{displaymath}
     We may, therefore, stick with the notation $\Rtx_>^k$ also when
     replacing $>$ by the Schreyer ordering $>_S$.

     Let $\varphi$ denote the substitution homomorphism
     \kurz{\begin{displaymath}
         \varphi: \Rtx_>^k\longrightarrow\Rtx_>^s:\varepsilon_i\mapsto
         g_i.
       \end{displaymath}}
\ifx\langform\ja
     \begin{center}
       \begin{tikzpicture}[description/.style={fill=white,inner sep=2pt}]
         \matrix (m) [matrix of math nodes, row sep=1em,column sep=0em,text height=1.5ex, text depth=0.25ex]
         { \varphi: \Rtx_>^k=\bigoplus_{i=1}^k \Rtx_>\cdot\varepsilon_i &\longrightarrow& \Rtx_>^s=\bigoplus_{j=1}^s \Rtx_> \cdot e_j, \\
           \phantom{\varphi: \Rtx_>^k=\bigoplus_{i=1}^k \Rtx_>\cdot\varepsilon_i} &\longmapsto & \phantom{\Rtx_>^s=\bigoplus_{j=1}^s \Rtx_> \cdot e_j} \\ };
         \node[anchor=east] at (m-2-1.east) {$\varepsilon_i$};
         \node[anchor=west] at (m-2-3.west) {$g_i.$};
       \end{tikzpicture}
     \end{center}
\fi
     We call its kernel the \emph{syzygy module} or simply the \emph{syzygies} of $G$,
     \begin{displaymath}
       \syz (G) := \left\{ \sum_{i=1}^k q_i\cdot\varepsilon_i \in \Rtx_{>_S}^k \suchthat \sum_{i=1}^k q_i\cdot g_i = 0 \right\}.
     \end{displaymath}
     The concept of syzygies is one that can be applied to any ring, and one of the conditions on our ground ring $R$ in Convention \ref{con:groundRing} states that we assume to be able to compute a finite system of generators for the syzygies of our leading coefficients,
     \begin{align*}
       &\syz_R(\lc_>(g_1),\ldots,\lc_>(g_k)) := \\
       &\qquad \{ (c_1,\ldots,c_k)\in R^k\mid c_1\cdot \lc_>(g_1)+\ldots+c_k\cdot \lc_>(g_k)=0 \}.
     \end{align*}
   \end{definition}

   In the case of a base field one constructs certain syzygies of
   a standard basis $G$ with the aid of s-polynomials
   in order to show  that $G$ is a standard basis. In order to
   treat the class of base rings introduced in Convention
   \ref{con:groundRing} we have to replace this set by a more subtle
   set of syzygies which we will now introduce. We will then show in
   Remark \ref{rem:comparison0} and Proposition \ref{prop:syzPID1}
   that in the case of a factorial
   base ring the new set of syzygies coincides with the classical
   one.

   \begin{definition}\label{def:syzTools}
     For a $k$-tuple $G=(g_1,\ldots,g_k)$ in $\Rtx^s$ and a
     fixed index $1\leq l \leq k$, we will now introduce several
     objects which will be of importance in the upcoming theory.

     Recall the notions of divisibility and least common multiple of
     module monomials in Definition \ref{def:divisibilityAndLCM}. We
     denote the set of least common multiples of the leading monomials
     up to and including $g_l$ with
     \begin{displaymath}
       C_l:=\left.\Big\{ \lcm(\LM_>(g_i)\mid i\in J) \suchthat
         J\subseteq\{1,...,k\} \text{ with } \max(J)=l \Big\}\right.\setminus\{0\} .
     \end{displaymath}
     Note that $C_l\subseteq \Rtx\cdot e_\lambda$ for the index $1\leq \lambda\leq s$ such that $\LT_>(g_l)\in \Rtx\cdot e_\lambda$.

     And for a least common multiple $a \in C_l$, we abbreviate the set of all indices $j$ up to $l$ such that $\lm_>(g_j)$ divides it with
     \begin{displaymath}
       J_{l,a} := \left.\Big\{ i\in\{1,\ldots,l\} \suchthat \lm_>(g_i) \text{ divides } a \Big\} \right. .
     \end{displaymath}

     Now given $J_{l,a}$, we can compute a finite generating set for
     the syzygies of the tuple $(\lc_>(g_i))_{i\in J_{l,a}}$, which we
     will temporarily denote with $S_R$. Let $\syz_{R,l,a}$
     be the set of elements of $S_R$ with non-trivial entry in $l$:
     \begin{center}
       \begin{tikzpicture}[description/.style={fill=white,inner sep=2pt}]
         \matrix (m) [matrix of math nodes, row sep=1em,column sep=-0.2em,text height=1.5ex, text depth=0.25ex]
         { \langle &S_R& \rangle_R = \Big\{ (c_i)_{i\in J_{l,a}} \in R^{|J_{l,a}|}\bigmid \sum\nolimits_{i\in J_{l,a}} c_i\cdot\lc_>(g_i)=0 \Big\}, \\ };
         \node [anchor=north,yshift=-0.5cm] (syz) at (m-1-2.south) {$\phantom{{}^{R,l,a}}\syz_{R,l,a}$};
         \node [anchor=west,xshift=-0.2cm] at (syz.east) {$=\Big\{ (c_i)_{i\in J_{l,a}} \in S_R\bigmid c_l\neq 0 \Big\}.\phantom{{}^{l,t^\beta x^\alpha}}$};
         \draw[draw opacity=0] (syz) -- node[sloped,xshift=-0.1cm] {$\subseteq$} (m-1-2);
       \end{tikzpicture}
     \end{center}

     With this, we can write down a finite set of syzygies of the leading terms of the $g_i$ up to and including $\lt_>(g_l)$ with non-trivial entry in $l$,
     \begin{displaymath}
       \syz_l := \left\{ \sum_{i\in J_{l,a}} \frac{c_i\cdot a}{\lm_>(g_i)}\cdot \varepsilon_i \in \Rtx^k \suchthat a\in C_l \text{ and } c\in\syz_{R,l,a} \right\}.
     \end{displaymath}

     For each $\xi'\in\syz_l$, we can then fix a single weak division with remainder of $\varphi(\xi')\in\Rtx^s$ with respect to $g_1,\ldots,g_l$ to obtain
     \begin{displaymath}
       \mathfrak{S}_l:=\left\{ u\cdot \xi'-\sum_{i=1}^k q_i\cdot\varepsilon_i \suchthat \begin{array}{c}
           \xi'\in \syz_l \text{ and } u\cdot \varphi(\xi') = q_1\cdot g_1+\ldots+q_l\cdot g_l + r \\ \text{the fixed weak division with remainder}
         \end{array}\right\}.
     \end{displaymath}
     As $\mathfrak{S}_l$ obviously depends on $G$, we write
     $\mathfrak{S}_{G,l}$ instead whenever $G$ is not clear from the
     context.  Moreover, we abbreviate
     \begin{displaymath}
       \mathfrak{S}^{(G)}:=\mathfrak{S}_{G,|G|}.
     \end{displaymath}
     Also, there is a certain degree of ambiguity in the construction
     of $\mathfrak S_l$ since we are actively choosing generating sets
     and divisions with remainders. Hence whenever we use $\mathfrak
     S_l$, it will represent any possible outcome of our
     construction. For example, when we write $\mathfrak S \subseteq
     \mathfrak{S}_l$ for a set $\mathfrak S\subseteq\Rtx^k_{>_S}$, it
     means that the elements of $\mathfrak S$ are possible outcomes of
     our construction of $\mathfrak{S}_l$.
   \end{definition}

   \begin{remark}[factorial ground rings]\label{rem:comparison0}
     Should $R$ be a factorial ring in which we have a natural notion
     of a least common multiple, then the construction above
     simplifies to extensions of classical techniques.

     Suppose $a\in C_l$ is a least common multiple of various leading
     monomials including $\lm_>(g_l)$. Let $J_{l,a}$ be the set of all
     indices $i$ for which $\lm_>(g_i)$ divides $a$. Then the syzygy
     module of all leading coefficients of $g_i$ with $i\in J_{l,a}$
     is generated by syzygies of the form (see Proposition
     \ref{prop:syzPID1})
     \begin{displaymath}
       \frac{\lcm(\lc_>(g_i),\lc_>(g_j))}{\lc_>(g_i)}\cdot \varepsilon_i - \frac{\lcm(\lc_>(g_i),\lc_>(g_j))}{\lc_>(g_j)}\cdot \varepsilon_j, \text{ with } i,j\in J_{l,a}, i>j.
     \end{displaymath}
     Abbreviating $\lambda_i:=\lc_>(g_i)$, we consequently get
     \begin{displaymath}
       \syz_{R,l,a} = \left\{ \frac{\lcm(\lambda_l,\lambda_i)}{\lambda_l}\cdot \varepsilon_l - \frac{\lcm(\lambda_l,\lambda_i)}{\lambda_i}\cdot \varepsilon_i \suchthat i\in J_{l,a} \right\}.
     \end{displaymath}
     Hence,
     \begin{displaymath}
       \syz_l = \bigcup_{a\in C_l} \left\{ \frac{\lcm(\lambda_l,\lambda_i)\cdot a}{\lt_>(g_l)}\cdot \varepsilon_l - \frac{\lcm(\lambda_l,\lambda_i)\cdot a}{\lt_>(g_i)}\cdot \varepsilon_i \suchthat i\in J_{l,a} \right\}.
     \end{displaymath}
     The definition of the Schreyer ordering $>_S$ now states
     \begin{displaymath}
       \lt_{>_S}\left(\frac{\lcm(\lambda_l,\lambda_i)\cdot a}{\lt_>(g_l)}\cdot \varepsilon_l - \frac{\lcm(\lambda_l,\lambda_i)\cdot a}{\lt_>(g_i)}\cdot \varepsilon_i\right) = \frac{\lcm(\lambda_l,\lambda_i)\cdot a}{\lt_>(g_l)}\cdot \varepsilon_l.
     \end{displaymath}
     Therefore, the module generated by the leading terms of $\syz_l$ is generated by the leading terms of its elements of the form
     \begin{displaymath}
       \frac{\lcm(\lt_>(g_l),\lt_>(g_i))}{\lt_>(g_l)}\cdot \varepsilon_l - \frac{\lcm(\lt_>(g_l),\lt_>(g_i))}{\lt_>(g_i)}\cdot \varepsilon_i \text{ with } l>i\in J_{l,a},
     \end{displaymath}
     which we obtain by setting $a=\lcm(\lm_>(g_l),\lm_>(g_i))$. Note that for $i\notin J_{l,a}$ the expression would just be zero.

     The images of these generators under $\varphi$ are, in the
     classical case of polynomial rings, commonly known as
     \emph{s-polynomials}, and the fixed divisions with remainder,
     which we considered for the definition of $\mathfrak S_l$,
     represent the normal form computations of these s-polynomials
     that are commonly done in the standard basis algorithm (and also
     Buchberger's Algorithm).
     We continue this train of thought in Remark
     \ref{rem:comparison1}.
   \end{remark}

   \begin{proposition}\label{prop:syzPID1} %Theorem 2.2.5 of \cite{Wienand11}
     Let $R$ be a factorial ring, and let $c_1,\ldots,c_k\in R$. Then
     \begin{displaymath}
       \syz(c_1,\ldots,c_k)=\left\langle \frac{\lcm(c_i,c_j)}{c_i}\cdot \varepsilon_i - \frac{\lcm(c_i,c_j)}{c_j}\cdot \varepsilon_j \suchthat k\geq i>j\geq 1 \right\rangle.
     \end{displaymath}
   \end{proposition}
   \begin{proof}
     We make an induction on $k$ with $k=1,2$ being clear. Now let $k>2$ and consider a syzygy $a:=a_1\cdot \varepsilon_1+\ldots+a_k\cdot \varepsilon_k$. Then
     \begin{displaymath}
       a_k\cdot c_k \in \langle c_1,\ldots,c_{k-1}\rangle,
     \end{displaymath}
     from which we can infer
     \begin{align*}
       a_k \in \langle c_1,\ldots,c_{k-1}\rangle : \langle c_k \rangle &= \langle c_1 \rangle : \langle c_k \rangle + \ldots + \langle c_{k-1}\rangle : \langle c_k \rangle \\
       &= \left\langle \frac{\lcm(c_1,c_k)}{c_k} \right\rangle + \ldots + \left\langle \frac{\lcm(c_{k-1},c_k)}{c_k} \right\rangle \\
     \end{align*}
     Setting
     \begin{displaymath}
       s_{ij}:=\frac{\lcm(c_i,c_j)}{c_i}\cdot \varepsilon_i - \frac{\lcm(c_i,c_j)}{c_j}\cdot \varepsilon_j \quad \text{and} \quad \mu_{ij} := \frac{\lcm(c_i,c_j)}{c_j},
     \end{displaymath}
     we have shown that there are $b_1,\ldots,b_{k-1}\in R$ such that
     \begin{displaymath}
       a_k = b_1\cdot \mu_{k1} + \ldots + b_{k-1}\cdot \mu_{kk-1},
     \end{displaymath}
     so that, by induction,
     \begin{align*}
       a-b_1\cdot s_{k1} + \ldots + b_{k-1}\cdot s_{kk-1} &\in \syz(c_1,\ldots,c_{k-1}) \\ &\quad = \langle s_{ij} \mid k-1\geq i > j \geq 1 \rangle.
     \end{align*}
     Hence,
     \begin{displaymath}
       a\in \langle s_{ij} \mid k-1\geq i > j \geq 1 \rangle + \langle s_{k1},\ldots,s_{kk-1}\rangle. \qedhere
     \end{displaymath}
   \end{proof}

   We now come back to the general case that $R$ is a  noetherian ring
   in which linear equations are solvable. For the objects in
   Definition \ref{def:syzTools} the following holds:

   \begin{lemma}\label{lem:ltS}
     For any $a\in C_l$ and any $(c_i)_{i\in J_{l,a}}\in syz_{R,l,a}$ there exists a $\xi \in \mathfrak{S}_l$ such that
     \begin{displaymath}
       \LT_{>_S}(\xi) = \frac{c_l\cdot a}{\LM_>(g_l)} \cdot \varepsilon_l.
     \end{displaymath}
   \end{lemma}
   \begin{proof} By construction in Definition \ref{def:syzTools}, for any $a\in C_l$ and any $(c_i)_{i\in J_{l,a}}\in syz_{R,l,a}$, there exists a $\xi\in \mathfrak{S}_l$ of the form
     \begin{displaymath}
       \xi = u\cdot \xi' - \sum_{i=1}^k q_i\cdot\varepsilon_i = \sum\nolimits_{i\in J_{l,a}}\frac{c_i\cdot a}{\lm_>(g_i)}\cdot\varepsilon_i -\sum_{i=1}^l q_i\cdot\varepsilon_i.
     \end{displaymath}

     First, recall that $J_{l,a}$ is the set of indices $i$ up to $l$ for which $\lm_>(g_i)$ divides $a$. Hence for all $i,j\in J_{l,a}$ we have
     \begin{displaymath}
       \lm_>\Big(\underbrace{\frac{c_i\cdot a}{\lm_>(g_i)}}_{\neq 0}\cdot g_i\Big)=a=\lm_>\Big(\underbrace{\frac{c_j\cdot a}{\lm_>(g_j)}}_{\neq 0}\cdot g_j\Big).
     \end{displaymath}
     As an immediate consequence, we get
     \begin{equation}
       \label{eq:lem1}
       \lt_{>_S}\Big(\sum\nolimits_{i\in J_{l,a}}\frac{c_i\cdot a}{\lm_>(g_i)}\cdot\varepsilon_i\Big) = \frac{c_l\cdot a}{\lm_>(g_l)}\cdot\varepsilon_l,
     \end{equation}
     because the Schreyer ordering prefers the highest component in case of a tie, and $l=\max J_{l,a}$, $c_l\neq 0$ by definition.

     Next, recall that $(c_i)_{i\in J_{l,a}}\in\syz_R(\lc_>(g_i)\mid i\in J_{l,a})$, which means that
     \begin{displaymath}
       \sum_{i\in J_{l,a}} \frac{c_i\cdot a}{\lm_>(g_i)}\cdot \LT_>(g_i)=\sum_{i\in J_{l,a}} c_i \lc_>(g_i)\cdot a \overset{!}{=} 0.
     \end{displaymath}
     Therefore, for all $j\in J_{l,a}$,
     \begin{displaymath}
       \lm_>\Big(\frac{c_j\cdot a}{\lm_>(g_j)}\cdot g_j\Big) > \lm_>\Big(\sum_{i\in J_{l,a}} \frac{c_i\cdot a}{\lm_>(g_i)}\cdot g_i\Big) = \lm_>(\varphi(\xi))
     \end{displaymath}
     as all summands have the same leading monomial $a$ and the leading terms in the sum cancel each other out.

     Finally, recall that $\varphi(\xi) = q_1\cdot g_1+\ldots+q_l\cdot g_l + r$ was a division with remainder, whose (ID1) property implies for all $j\in J_{l,a}$ and $i=1,\ldots,l$
     \begin{displaymath}
       \lm_>\Big(\frac{c_j\cdot a}{\lm_>(g_j)}\cdot g_j\Big) > \lm_>(\varphi(\xi))\overset{\text{(ID1)}}{\geq}\lm_>(q_i\cdot g_i).
     \end{displaymath}
     Thus we have for all $j\in J_{l,a}$ and $i=1,\ldots,l$
     \begin{equation}
       \label{eq:lem2}
       \lm_{>_S}\Big(\frac{c_j\cdot a}{\lm_>(g_j)}\cdot \varepsilon_j\Big) >_S \lm_{>_S}(q_i\cdot \varepsilon_i).
     \end{equation}

     Together, we obtain
     \begin{align*}
       \lt_{>_S}(\xi) & \overset{\phantom{(\ref{eq:lem2})}}{=} \lt_{>_S}\Big(u\cdot \sum_{j\in J_{l,a}}\frac{c_j\cdot a}{\lm_>(g_j)}\cdot\varepsilon_j -\sum_{i=1}^l q_i\cdot\varepsilon_i\Big) \\[0.5em]
       & \overset{(\ref{eq:lem2})}{=} \lt_{>_S}\Big(u\cdot  \sum_{j\in J_{l,a}}\frac{c_j\cdot a}{\lm_>(g_j)}\cdot\varepsilon_j\Big) \overset{(\ref{eq:lem1})}{=} \frac{c_l\cdot a}{\lm_>(g_l)}\cdot\varepsilon_l. \qedhere
     \end{align*}
   \end{proof}

   \begin{theorem}\label{thm:standardBasis}
     Let $G=(g_1,\ldots,g_k)$ be a $k$-tuple of elements in $\Rtx^s$
      and let $\mathfrak{S}_1,\ldots,\mathfrak{S}_k$ be
     constructed as in Definition \ref{def:syzTools}. Suppose there exists an
     $\mathfrak{S}\subseteq\bigcup_{l=1}^k \mathfrak{S}_l$ such that
     $\LT_{>_S}(\mathfrak{S})=\LT_{>_S}(\bigcup_{l=1}^k \mathfrak{S}_l)$
     and $\varphi(\xi) = 0$ for all $\xi\in \mathfrak{S}$.
     Then
     $G$ is a standard basis  with respect to $>$
     and $\mathfrak{S}$ is a standard basis of $\syz(G)$ with respect to $>_S$.
   \end{theorem}
   \begin{proof}
     Let $q_1,\ldots,q_k\in\Rtx_>=\Rtx_{>_S}$ be chosen
     arbitrarily. We will proof both statements simultaneously via the
     standard representation criteria in Proposition
     \ref{prop:qbuchberger0} (c), by considering
     \begin{displaymath}
       \chi := \sum_{i=1}^k q_i\cdot \varepsilon_i \quad \text{ and } \quad g := \varphi(\chi)=\sum_{i=1}^k q_i\cdot g_i.
     \end{displaymath}
     % Since $g$ represents an arbitrary element of $M$, it will yield the result for $M$, and since $\chi$ represents, for $g=0$, an arbitrary element of $\syz(G)$,  it will yield the result for $\syz(G)$.
     Here $g$ represents an arbitrary element of $M$, and, in case $g=0$, $\chi$ represents an arbitrary element of $\syz(G)$.

     First compute a division with remainder of $\chi$ with respect to $\mathfrak{S}$ and the Schreyer ordering,
     \begin{displaymath}
       \chi = \sum_{\xi\in \mathfrak{S}} a_{\xi} \cdot \xi + r.
     \end{displaymath}
     Should $r$ be zero, then the expression above is a standard
     representation of $\chi$ with respect to $>_S$. Moreover, as
     $\varphi(\xi)=0$ for all $\xi\in \mathfrak{S}$ by assumption,
     $g=\varphi(\chi)=0$ trivially possesses a standard
     representation. Hence, in case $r=0$, both $g$ and $\chi$ satisfy
     the standard representation criteria. So suppose $r\neq 0$ for
     the remainder of the proof.

     By Corollary \ref{cor:sid2>}, we may assume that our division with remainder satisfies (SID2), i.e. say
     \begin{equation}
       r = r_1\cdot\varepsilon_1+\ldots+r_k\cdot\varepsilon_k \text{ with } \LT_>(r_i\cdot\varepsilon_i)\notin \LT_{>_S}(\mathfrak{S}) \text{ for all } i=1,\ldots,k.
       \label{eq:Schreyer1}
     \end{equation}
     Since by assumption $\varphi(\xi)=0$ for all $\xi\in \mathfrak{S}$, we have
     \begin{equation}
       g = \varphi(\chi) = \varphi(r) = r_1\cdot g_1+\ldots+r_k\cdot g_k.
       \label{eq:Schreyer2}
     \end{equation}
     To proof the statement for $G\subseteq M$, it suffices to show
     that the expression above is a standard representation of $g$. To
     proof the statement for $\mathfrak{S}\subseteq\syz(G)$, we will
     show that $r\neq 0$ contradicts $g=0$. This leaves $r=0$ as the
     only viable case, assuming $g=0$, for which we have already
     established that $\chi$ satisfies the standard representation
     criteria.

     Now assume that $\lm_>(g)<\lm_>(r_i\cdot g_i)$ for some
     $i=1,\ldots,k$, and hence for $J:=\{i\in\{1,\ldots,k\}\mid
     \lm_>(r_i\cdot g_i) \text{ maximal}\}$
     \begin{displaymath}
       \sum_{i\in J} \lt_>(r_i\cdot g_i) = 0.
     \end{displaymath}
     Set $l:=\max(J)$ and $a:=\lcm(\lm_>(g_i)\mid i\in J)$, so that
     obviously $J\subseteq J_{l,a}$. We will now concentrate on
     $r_l\cdot \varepsilon_l$.

     For the leading coefficient of $r_l\cdot \varepsilon_l$, note
     that the leading coefficients sum up to zero, i.e. $\sum_{i\in J}
     \lc_>(r_i)\cdot\varepsilon_i\in\syz(\lc_>(g_i)\mid i\in
     J_{l,a})$. Recall that $\syz_{R,l,a}$ are the elements of a
     generating system of $\syz(\lc_>(g_i)\mid i\in J_{l,a})$ with
     non-trivial entry in $l$. Hence there are suitable $d_{(c_i)}\in
     R$ such that
     \begin{equation}
       \lc_>(r_l)\cdot \varepsilon_l = \sum_{(c_i)\in\syz_{l,a}} d_{(c_i)} \cdot c_l\cdot \varepsilon_l.
       \label{eq:Schreyer3}
     \end{equation}
     For the leading monomial of $r_l\cdot \varepsilon_l$, note that
     $\lm_>(r_l\cdot g_l)$ is divisible by $\lm_>(g_i)$ for all $i\in
     J$. Hence it is divisible by $a=\lcm(\lm_>(g_i)\mid i\in J)$,
     i.e. there exists a $t^\delta x^\gamma$ such that $\lm_>(r_l\cdot
     g_l)= t^\delta x^\gamma \cdot a$, or equivalently
     \begin{equation}
       \lm_>(r_l) = t^\delta x^\gamma \cdot \frac{a}{\lm_>(g_l)}.
       \label{eq:Schreyer4}
     \end{equation}
     Now, by the previous Lemma \ref{lem:ltS} there exists a $\xi_{(c_i)}\in \mathfrak{S}_l$ for any $(c_i)\in\syz_{R,l,a}$ such that
     \begin{equation}
       \LT_{>_S}(\xi_{(c_i)}) = \frac{c_l\cdot a}{\LM_>(g_l)}\cdot \varepsilon_l.
       \label{eq:Schreyer5}
     \end{equation}
     Piecing everything together, we thus get
     \begin{align*}
       \lt_>(r_l)\cdot\varepsilon_l &\overset{(\ref{eq:Schreyer3})+(\ref{eq:Schreyer4})}{=} t^\delta x^\gamma \sum_{(c_i)\in\syz_{l,a}} d_{(c_i)}\cdot \frac{c_l\cdot a}{\lm_>(g_l)} \cdot \varepsilon_l \\
       &\underset{\phantom{(\ref{eq:Schreyer3})+(\ref{eq:Schreyer4})}}{\overset{(\ref{eq:Schreyer5})}{=}} t^\delta x^\gamma \sum_{(c_i)\in\syz_{l,a}} d_{(c_i)}\cdot \LT_{>_S}(\xi_{(c_i)}) \in \LT_{>_S}(\mathfrak{S}_l).
     \end{align*}
     And since $\LT_{>_S}(\mathfrak{S}_l)\subseteq
     \LT_{>_S}(\mathfrak{S})$ by our first assumption, this
     contradicts the (SID2) condition in Equation
     (\ref{eq:Schreyer1}). Therefore, Equation (\ref{eq:Schreyer2})
     has to be a standard representation, implying that $G$ is a
     standard basis of $M$ with respect to $>$.

     Moreover, since $r\neq 0$, Equation (\ref{eq:Schreyer2}) being
     standard representation yields an obvious contradiction if
     $g=0$. Hence in the case $g=0$, we have $r=0$ and we have already
     seen how this implies that $\mathfrak{S}$ is a standard basis of
     $\syz(G)$ with respect to $>_S$.
   \end{proof}

   \begin{remark}[factorial rings continued]\label{rem:comparison1}
     Suppose again that $R$ is a factorial ring. We have seen in
     Remark \ref{rem:comparison0}, that the leading module of
     $\bigcup_{l=1}^k\mathfrak S_{G,l}$ is generated by the leading
     terms of elements of the form
     \begin{displaymath}
       \frac{\lcm(\lt_>(g_i),\lt_>(g_j))}{\lt_>(g_i)}\cdot \varepsilon_i - \frac{\lcm(\lt_>(g_i),\lt_>(g_j))}{\lt_>(g_j)}\cdot \varepsilon_j, \; i>j.
     \end{displaymath}
     They are, thus,  the only elements we need to keep track of for Theorem \ref{thm:standardBasis}.
     These elements are obviously characterized by pairs of distinct
     elements $(g_i,g_j)$ that is, by elements in a so-called
     \emph{pair-set}, which commonly appear in the classical standard
     basis algorithm and in Buchberger's Algorithm.
   \end{remark}

   \begin{algorithm}[standard basis algorithm]\label{alg:standardBases} \skipalgorithm
     \begin{algorithmic}[1]
       \REQUIRE $(G,>)$, where $G=(g_1,\ldots,g_k)$ be a $k$-tuple of
       elements in $\Rtx^s$ generating $M\leq\Rtx^s$ and $>$ a $t$-local monomial ordering on $\Mon^s(t,x)$.
       \ENSURE $G'\subseteq M$ a standard basis of $M$ with respect to $>$.
%       \STATE Suppose $G:=\{g_1,\ldots,g_k\}$.
       \STATE Pick $\mathfrak{S}\subseteq\bigcup_{l=1}^k \mathfrak{S}_{G,l}\subseteq\Rtx^{k}$ such that
       \begin{displaymath}
         \LT_{>_S}(\mathfrak{S})=\LT_{>_S}\Big(\bigcup_{l=1}^k \mathfrak{S}_{G,l}\Big),
       \end{displaymath}
       where $>_S$ is the Schreyer ordering on $\Mon^{k}(t,x)$ associated to $G$ and $>$.
       \WHILE{$\mathfrak{S}\neq\emptyset$}
       \STATE Set $k:=|G|$, so that $G:=\{g_1,\ldots,g_k\}$ and $\mathfrak{S}\subseteq\Rtx_>^k$.
       \STATE Choose $q=\sum_{i=1}^k q_i\cdot\varepsilon_i \in \mathfrak{S}$.
       \STATE Set $\mathfrak{S}:=\mathfrak{S}\setminus\{q\}$.
       \STATE Compute a weak normal form $r$ of $q_1\cdot g_1+\ldots+q_k\cdot g_k$ with respect to $G$
       \begin{displaymath}
         r:=\NF_>(q_1\cdot g_1+\ldots+q_k\cdot g_k, G,>).
       \end{displaymath}
       \IF{$r\neq 0$}
       \STATE Set $g_{k+1}:=r$.
       \STATE Set $G:=G\cup\{g_{k+1}\}$.
       \STATE Pick $\mathfrak{S}'\subseteq \mathfrak{S}^{(G)}\subseteq\Rtx^{k+1}$ such that
       \begin{displaymath}
         \LT_{>_S}(\mathfrak{S}')=\LT_{>_S}\Big(\mathfrak{S}^{(G)}\Big),
       \end{displaymath}
       where $>_S$ is the Schreyer ordering on $\Mon^{k+1}(t,x)$ induced by the newly extended $G$ and $>$.
       \STATE Set $\mathfrak{S}:= (\mathfrak{S}\!\times\!\{0\}) \cup \mathfrak{S}'$.
       \ENDIF
       \ENDWHILE
       \RETURN{$G$.}
     \end{algorithmic}
   \end{algorithm}
   \begin{proof} Label all objects in the $\nu$-th iteration of the while loop with a subscript $\nu$. That is, to be more precise,
     \begin{itemize}[leftmargin=*]
     \item $G_\nu$ as it exists in Step $4$,
     \item $k_\nu$ as it exists in Step $4$,
     \item $q_\nu$ as chosen in Step $5$
     \item $r_\nu$ as computed in Step $7$,
     \item $\mathfrak S_\nu$ as $\mathfrak S$ exists in Step $4$,
     \item $\mathfrak S'_{\nu+1}$ as $\mathfrak S'$ exists in Step $9$ if $r_{\nu-1}\neq 0$, $\mathfrak S'_{\nu+1}=\emptyset$ otherwise, $\mathfrak S_1':=\mathfrak S_1$,
     \end{itemize}
     so that
     \begin{displaymath}
       G_{\nu+1} = G_{\nu} \cup \{ r_\nu\} \text{ and } \mathfrak{S}_{\nu+1} = (\mathfrak{S}_\nu \!\times\!\{0\}) \cup \mathfrak{S}'_{\nu+1}.
     \end{displaymath}

     \emph{Termination.}
     Note that we have a nested sequence of modules
     \begin{displaymath}
       \LT_>(G_{1})\subseteq\LT_>(G_{2})\subseteq\LT_>(G_{3})\subseteq \hdots \subseteq \LT_>(G_{\nu})\subseteq\LT_>(G_{\nu+1}) \subseteq \hdots,
     \end{displaymath}
     which has to stabilize at some point.
     Because $r_\nu\neq 0$ implies
     $\LT_>(G_{\nu})\subsetneq\LT_>(G_{\nu+1})$, it means that our
     sets $\mathfrak S_\nu$ have to be strictly decreasing in every
     step beyond the point of stabilization.
     And since all $\mathfrak S_\nu$ are finite, our algorithm terminates eventually.

     \emph{Correctness.} Let $N$ be the total number of iterations, and let $G$ be the return value, $k:=|G|$.
     We will prove that $G$ is a standard basis by constructing a set
     $\mathfrak S\subseteq\Rtx^k$ that satisfies the two conditions in
     Theorem \ref{thm:standardBasis}. For that, consider all
     $\mathfrak S_\nu\subseteq\Rtx_>^{k_\nu}$ canonically embedded in
     $\Rtx_>^k$ due to $G_\nu\subseteq G$ and $k_\nu\leq k$. Let
     $\mathfrak S$ be the union of all $\mathfrak S_\nu'$,
     \begin{displaymath}
       \mathfrak S:=\bigcup_{\nu=1}^{N+1} \mathfrak S_\nu' \subseteq \Rtx^k.
     \end{displaymath}

     Note that $\mathfrak S_\nu'\subseteq \mathfrak S_{G,k_\nu}$,
     because the construction of $\mathfrak S_{G,k_\nu}$ only depends
     on the first $k_\nu$ elements of $G$, which are exactly the
     elements of $G_\nu$.
     Moreover, Step $9$ implies that $\lt_{>_S}(\mathfrak S_\nu) =
     \LT_{>_S}(\mathfrak S_{G,k_\nu})$, which shows that $\mathfrak S$
     satisfies the first condition of our theorem,
     \begin{displaymath}
       \lt_{>_S}(\mathfrak S)= \lt_{>_S}\left(\,\bigcup_{l=1}^{k} \mathfrak S_{G,l}\right).
     \end{displaymath}

     Now for each $\xi\in\mathfrak S$ there exists an iteration
     $1\leq\nu\leq N$ in which it is chosen in Step $5$, $\xi =
     \sum_{i=1}^{k_\nu} q_{i,\nu}\cdot\varepsilon_i$.

     If $\varphi(\xi)=r_\nu=0$, then $\xi$ satisfies the second
     condition of our theorem. However if $\varphi(\xi)=r_\nu\neq 0$,
     then $g_{\nu+1} = r_\nu$ and $\xi$ can be replaced with
     $\xi-\varepsilon_{\nu+1}$ so that
     $\varphi(\xi-\varepsilon_{\nu+1})=0$. Note that this does not
     change the leading term, since by construction the maximal
     leading terms of $q_1\cdot g_1, \ldots, q_{l_\nu}\cdot g_{l_\nu}$
     cancel each other out, which implies that
     $q_{i,\nu}\cdot\varepsilon_i >_S\varepsilon_{\nu+1}$ for any
     $1\leq i \leq \nu$ with $q_{i,\nu}\neq 0$. Hence we obtain a set
     $\mathfrak S$ completely satisfying the second condition of our
     theorem.
   \end{proof}

   \begin{remark}[polynomial input]\label{rem:polynomialStandardBases}
     Should our input be polynomial, $g_1,\ldots,g_k\in R[t,x]^s$, then all
     normal form computations terminate and yield polynomial outputs
     as noted in \ref{rem:polynomialNormalForm}. In particular, our
     standard basis algorithm will terminate and the output will be
     polynomial as well.

     Moreover, if our input is $x$-homogeneous, then so is the resulting standard basis.
   \end{remark}

   Should $R$ be a factorial ring, Algorithm \ref{alg:standardBases} can be simplified to:

   \begin{algorithm}[standard basis algorithm for factorial
     rings]\label{alg:standardBasesPrincipal} \skipalgorithm
     \begin{algorithmic}[1]
       \REQUIRE $(G,>)$, where $G=(g_1,\ldots,g_k)$ be a $k$-tuple of
       elements in $\Rtx^s$ generating $M\leq\Rtx^s$ with $R$
       factorial and $>$ a
       $t$-local monomial ordering on $\Mon^s(t,x)$.
       \ENSURE $G'\subseteq M$ a standard basis of $M$ with respect to $>$.
       \STATE Suppose $G:=\{g_1,\ldots,g_k\}$.
       \STATE Initialize a pair-set, $P:=\{(g_i,g_j)\mid i<j\}$.
       \WHILE{$P\neq\emptyset$}
       \STATE Pick $(g_i,g_j)\in P$.
       \STATE Set $P:=P\setminus\{(g_i,g_j)\}$.
       \STATE Compute a weak normal form
       \begin{displaymath}
         r:=\NF_>(\spoly(g_i,g_j), G,>),
       \end{displaymath}
       where
       \begin{align*}
         &\spoly(g_i,g_j) \\
         &\qquad =\frac{\lcm(\lt_>(g_i),\lt_>(g_j))}{\lt_>(g_i)}\cdot g_i - \frac{\lcm(\lt_>(g_i),\lt_>(g_j))}{\lt_>(g_j)}\cdot g_j
       \end{align*}
       and
       \begin{align*}
         &\lcm(\lt_>(g_i),\lt_>(g_j)) \\
         &\qquad = \lcm(\lc_>(g_i),\lc_>(g_j))\cdot \lcm(\lm_>(g_i),\lm_>(g_j)).
       \end{align*}
       \IF{$r\neq 0$}
       \STATE Extend the pair-set, $P:=P\cup\{(g,r)\mid g\in G\}$.
       \STATE Set $G:=G\cup\{r\}$.
       \ENDIF
       \ENDWHILE
       \RETURN{$G':=G$.}
     \end{algorithmic}
   \end{algorithm}

\lang{
  \section{Standard basis algorithm for an application in tropical geometry}
}

   \begin{remark}[simplification for ideals in tropical geometry]\label{rem:tropicalcase}
     The most important application of standard bases over rings that we
     have in mind is motivated by tropical geometry over the field of
     $p$-adic numbers $\QQ_p$. Given a homogeneous ideal in $\QQ_p[x]$ we have to
     decide if the initial ideal with respect to some weight vector
     $w\in\RR^n$ is monomial free or not, where for the initial forms the
     valuation of the coefficients is taken into account (see \cite[Chapter~2]{MS15}). For this the
     ideal can be restricted to $\ZZ_p[x]$ and via the surjection
     \begin{displaymath}
       \pi:\ZZ\llbracket t\rrbracket[x]\longrightarrow\ZZ_p[x]:t\mapsto p
     \end{displaymath}
     we may pull the ideal back to the mixed power series ring
     $\ZZ\llbracket t\rrbracket[x]$. It is not hard to see
     (\cite{MR15b}) that the initial ideal
     of $I=\langle f_1,\ldots,f_k\rangle\unlhd\QQ_p[x]$ with respect to $w$ with
     $f_i\in\ZZ[x]$ is monomial free if and only if the initial ideal
     with respect to $(-1,w)$ of
     \begin{displaymath}
       J=\langle p-t,f_1,\ldots,f_k\rangle \unlhd \ZZ\llbracket t\rrbracket[x]
     \end{displaymath}
     is monomial free. But this can be
     checked through repeated standard basis computations. We are, thus,
     particularly interested in computing standard bases of
     $x$-homogeneous ideals in $\ZZ\llbracket t\rrbracket[x]$
     containing $p-t$ for some prime number $p$. Due to practical constraints,
     we restrict ourselves to ideals generated by polynomials.

     This is a situation that can be heavily exploited for our division algorithms. For any
     polynomial $f$ occuring in the reduction process either the leading
     coefficient $c$ is divisible by $p$ and can thus be reduced by $p-t$, or
     it is coprime to $p$, in which case the Euclidean Algorithm
     provides integers $a,b\in\ZZ$ such that
     \begin{displaymath}
       1=a\cdot c+b\cdot p,
     \end{displaymath}
     and hence replacing $f$ by $a\cdot f+b\cdot \lm_>(f)\cdot (p-t)$ we can pass to a
     polynomial with leading coefficient $1$. If we preprocess all
     polynomials, except $p-t$, added to our standard basis in the
     standard basis algorithm that way ($g_1,\ldots,g_k$ in the Input and $g_{k+1}$ in Step $9$ of Algorithm \ref{alg:standardBases}), checking if a leading term can
     be reduced (Step $3$ in Algorithms \ref{alg:HDDwR} and \ref{alg:DwR}) burns down to a simple divisibility check as in the case
     of standard bases over fields.
   \end{remark}

\lang{

  We will now describe the algorithms for the special case described
  in Remark~\ref{rem:tropicalcase} in detail, starting with the
  algorithm reducing a polynomial with respect to $p-t$.

   \begin{algorithm}[$\pRed$ --- $(p-t)$-reduce]\label{alg:pReduce} \skipalgorithm
     \begin{algorithmic}[1]
       \REQUIRE{$(g,>)$, where $>$ is a $t$-local monomial ordering and $g\in\ZZ[t,x]$.}
       \ENSURE{$(a,q,r)$ with $a\in\{1,\ldots,p-1\}$ and $q,r\in\ZZ[t,x]$,
         such that $a\cdot g=q\cdot (p-t)+r$,
         $\lm_>(g)\geq\lm_>(q)$ and either $r=0$ or
         $\lc_>(r)=1$.}
       \STATE Set $q:=0$
       \STATE Set $r:=g$.
       \WHILE{$p\;\mid\;\lc_>(r)$}
       \STATE Let $l:=\max\{m\in\NN\mid p^m \text{ divides } \lc_>(r)\}>0$.
       \STATE Set $r:=r-\frac{\lt_>(r)}{p^l}\cdot (p^l-t^l)$.
       \STATE Set $q:=q+\frac{\lt_>(r)}{p^l}\cdot\frac{p^l-t^l}{p-t}$.
       \ENDWHILE
       \IF{$r\not=0$}
       \STATE Compute with the Euclidean Algorithm
       $a\in\{1,\ldots,p-1\}$ and $b\in\ZZ$ such that
       $1=a\cdot\lc_>(r)+b\cdot p$.
       \STATE Set $r:=a\cdot r+b\cdot (p-t)\cdot \lm_>(r)$.
       \STATE Set $q:=a\cdot q-b\cdot \lm_>(r)$.
       \ENDIF
       \RETURN{$(a,q,r)$}
     \end{algorithmic}
   \end{algorithm}
   \begin{proof}
     \emph{Termination:} We need to show that eventually $p$ does not
     divide the leading coefficient of $r$ anymore.
     Let us for a moment consider the polynomial
     \begin{displaymath}
       r=\sum_{i=1}^k r_i\cdot x^{\alpha_i}
     \end{displaymath}
     as a polynomial
     in $x$ with coefficients $r_i$ in $\ZZ[t]$.
     Then the set of monomials
     in $x$ occuring in $r$ does not increase throughout the
     algorithm. Moreover, if the leading monomial of $r$ is contained
     in $r_i\cdot x^{\alpha_i}$ with
     \begin{displaymath}
       r_i=c_{i_1}\cdot t^{i_1}+\ldots+c_{i_j}\cdot t^{i_j}, i_1<\ldots<i_j,
     \end{displaymath}
     then in Step $5$ we
     substitute the term $c_{i_1}\cdot t^{i_1}x^{\alpha_i}$ by the term
     $c_{i_1}/p^l \cdot t^{i_1+l}x^{\alpha_i}$, increasing the
     minimal $t$-degree in $r_i$ strictly.

     Let $\nu_p(c):=\max\{m\in\NN\mid p^m \text{ divides } c\}$ denote
     the $p$-adic valuation on $\ZZ$, so that $l=\nu_p(c_{i_1})$, and
     consider the valued degree of $r_i$ defined by
     \begin{displaymath}
       m_i:=\max\{ \nu_p(c_{i_1})+\deg(t^{i_1}),\ldots, \nu_p(c_{i_j})+\deg(t^{i_j}) \}.
     \end{displaymath}
     This is a natural upper bound on the $t$-degree of our
     substituted $r_i$, and hence
     \begin{displaymath}
       \max\{m_1,\ldots,m_k\}
     \end{displaymath}
     is an upper bound for the $t$-degree of all terms in our new $r$.

     If the monomial of the substitute, $t^{i_1+l}x^{\alpha_i}$, does not
     occur in the original $r$, then this upper bound remains the same
     for out new $r$. If it does occur in the original $r$, then this
     valued degree might increase depending on the sum of the
     coefficients, however the number of terms in $r$ strictly
     decreases.

     Because $r$ has only finitely many terms to begin with, this upper
     bound may therefore only increase a finite number of times. And
     since the minimal $t$-degree is strictly increasing, if $p$
     divides the leading coefficient of $r$, our algorithm terminates
     eventually.

     \emph{Correctness:} Once the while loop is done, we have found
     polynomials $q$ and $r$ such that
     $g=q\cdot (p-t)+r$ and $\lm_>(g)\geq\lm_>(q)$.
     Moreover, we may assume that $r\not=0$. Since $p$
     does not divide the leading coefficient of $r$, these numbers are
     coprime and the Euclidean Algorithm computes integers $a,b\in\ZZ$
     such that
     \begin{displaymath}
       1=a\cdot \lc_>(r)+b\cdot p,
     \end{displaymath}
     and we may assume $a\in\{1,\ldots,p-1\}$. This leads to the equation
     \begin{displaymath}
       a\cdot g=(a\cdot q-b\cdot \lm_>(r))\cdot (p-t)+(a\cdot
       r+b\cdot (p-t)\cdot\lm_>(r)),
     \end{displaymath}
     and we are done by replacing $q$ with $a\cdot q-b\cdot \lm_>(r)$
     and $r$ with $a\cdot
     r+b\cdot (p-t)\cdot\lm_>(r)$. It is clear by construction that then
     $\lm_>(g)\geq\lm_>(q)$ and $\lc_>(r)=1$.
   \end{proof}

   \begin{remark}
     Given $p-t$ and a polynomial $g$ as in
     Algorithm~\ref{alg:pReduce}, we are interested in the ideal
     generated by these in the ring $\ZZ\llbracket t\rrbracket[x]$. If $r$ is the output
     of Algorithm~\ref{alg:pReduce}, then we have indeed
     \begin{displaymath}
       \langle p-t,g\rangle=\langle p-t,r\rangle \unlhd\ZZ\llbracket t\rrbracket[x].
     \end{displaymath}
     To see this consider the equation
     \begin{displaymath}
       a\cdot g=s\cdot (p-t)+r
     \end{displaymath}
     which implies the inclusion $\supseteq$. For the other inclusion
     it suffices to note that the integer $a\in\{1,\ldots,p-1\}$ is a unit in
     the ring of $p$-adic numbers $\ZZ\llbracket t\rrbracket/\langle
     p-t\rangle\cong\ZZ_p$.

     Moreover, note that the polynomials $q$ and $r$ will be
     $x$-homogeneous, if the input $g$ was $x$-homogeneous.
   \end{remark}

   Next we adjust the homogeneous determinate division with remainder
   to the situation that all but the first element in $G$ have leading
   coefficient one. This will be formulated for any base ring as in
   Convention~\ref{con:groundRing} and for any finite number of
   $t$-variables and $x$-variables.

   \begin{algorithm}[$\SHDDwR$ --- special version]\label{alg:tropHDDwR}\skipalgorithm
     \begin{algorithmic}[1]
       \REQUIRE{$(f,G,>)$, where $f\in\Rtx^s$ $x$-homogeneous,
         $G=(g_1,\ldots,g_k)$ a $k$-tuple of $x$-homogeneous elements
         in $\Rtx^s$ with $g_1=p-t$ and $\lc_>(g_i)=1$ for $i=2,\ldots,k$ and $>$ a $t$-local
         monomial ordering on $\Mon^s(t,x)$.}
       \ENSURE{$(Q,r)$, where $Q=(q_1,\ldots,q_k)\in\Rtx^k$ and $r\in \Rtx^s$ such that
         \begin{displaymath}
           f=q_1\cdot g_1+\hdots+q_k\cdot g_k+r
         \end{displaymath}
         satisfies
         \begin{description}
         \item[\rm (DD1)] no term of $q_i\cdot\lt_>(g_i)$ lies in $\langle\lt_>(g_j)\mid j<i\rangle$ for all $i$,
         \item[\rm (DD2)] no term of $r$ lies in $\langle\lt_>(g_1),\ldots,\lt_>(g_k)\rangle$,
         \item[\rm (DDH)] the $q_1,\ldots,q_k,r$ are either $0$ or $ x$-homogeneous of $ x$-degree \linebreak
           $\deg_ x(f)-\deg_ x(g_1),\ldots,\deg_ x(f)-\deg_ x(g_k),\deg_ x(f)$ respectively.
         \end{description}}
       \STATE Set $q_i:= 0$ for $i=1,\ldots,k$, $r:=0$, $\nu:=0$, $f_\nu:=f$.
       \WHILE{$f_\nu\neq 0$}
       \IF{$\exists\;i\;:\;\lt_>(g_i)\;\mid\;\lt_>(f_\nu)$}
       \STATE Choose $i\in\{1,\ldots,k\}$ minimal with $\lt_>(g_i)\;\mid\;\lt_>(f_\nu)$.
       \FOR{j=1,\ldots,k}
       \STATE Set
       \begin{displaymath}
         q_{j,\nu}:=
         \begin{cases}
           \frac{\lt_>(f_\nu)}{\lt_>(g_i)}, & \text{if } j=i, \\
           0,                               & \text{otherwise.}
         \end{cases}
       \end{displaymath}
       \ENDFOR
       \STATE Set $r_\nu:=0$.
       \ELSE
       \STATE Set $q_{i,\nu}:=0$, for $i=1,\ldots,k$, and $r_\nu:=\lt_>(f_\nu)$.
       \ENDIF
       \STATE Set $q_i:=q_i + q_{i,\nu}$ for $i=1,\ldots,k$ and $r:= r+r_\nu$.
       \STATE Set $f_{\nu+1}:=f_{\nu} - (q_{1,\nu}\cdot g_1 + \ldots + q_{k,\nu}\cdot g_k + r_\nu)$ and $\nu:=\nu+1$.
       \ENDWHILE
       \RETURN{$((q_1,\ldots,q_k),r)$}
     \end{algorithmic}
   \end{algorithm}
   \begin{proof}
     We just have to note that the condition
     \begin{displaymath}
       \lt_>(f_\nu)\in\langle \lt_>(g_1),\ldots,\lt_>(g_k)\rangle
     \end{displaymath}
     is equivalent to the condition
     \begin{displaymath}
       \exists\;i\;:\;\lt_>(g_i)\;\mid\;\lt_>(f_\nu).
     \end{displaymath}
     For this observe, that as soon as some $\lt_>(g_i)$ for
     $i=2,\ldots,k$ occurs in a linear combination representing
     $\lt_>(f_\nu)$ then necessarily $\lt_>(g_i)$ divides
     $\lt_>(f_\nu)$.

     Hence, the algorithm coincides with Algorithm~\ref{alg:HDDwR},
     only the test in Step $3$ has been simplified.
   \end{proof}

   \begin{remark}
     In Remark~\ref{rem:tropicalcase} we explained that our main
     interest lies in the computation of a standard basis for ideals
     genereated by $p-t$ and a finite number of $x$-homogeneous
     polynomials. For this we only need a suitable division algorithm
     and one might think, that $\SHDDwR$ applies in that
     situation. However, it does not! The problem here is termination.
     If we try to reduce completely with respect to $p-t$ the
     algorithm will in general produce power series and will not
     terminate. In our application we will have to consider $t$ as an additional
     polynomial variable. Then the division algorithm $\SDwR$ (see Algorithm~\ref{alg:tropDwR})
     applies and terminates.
   \end{remark}

   In the specialized algorithm for weak division with remainder we
   restrict to the base ring $\ZZ$. Moreover, we assume that the input
   is polynomial, so that we are able to homogenize also with respect
   to the variable $t$. We, therefore, change our convention for this
   one algorithm and set $x=(t,x_1,\ldots,x_n)$.

   \begin{algorithm}[$\SDwR$ - special version of $\DwR$] \label{alg:tropDwR}\skipalgorithm
     \begin{algorithmic}[1]
       \REQUIRE{$(f,G,>)$, where $f\in\ZZ[x]=\ZZ[t,x_1,\ldots,x_n]$ and $G=(g_1,\ldots,g_k)$
         is a $k$-tuple in $\ZZ[x]$ with $g_1=p-t$ and $\lc_>(g_i)=1$
         for $i=2,\ldots,k$ and $>$ a $t$-local
         monomial ordering on $\Mon(x)=\Mon(t,x_1,\ldots,x_n)$.}
       \ENSURE{$(u,Q,r)$, where $u\in \ZZ[x]$ with $p\nmid\lc_>(u)=\lt_>(u)$,
         $Q=(q_1,\ldots,q_k)\subseteq\ZZ[x]^k$
         and $r\in \ZZ[x]$ such that
         \begin{displaymath}
           u\cdot f=q_1\cdot g_1+\hdots+q_k\cdot g_k+r
         \end{displaymath}
         satisfies
         \begin{description}
         \item[\rm (ID1)] $\lm_>(f)\geq\lm_>(q_i\cdot g_i)$ for $i=1,\ldots,k$ and
         \item[\rm (ID2)] $\lt_>(r)\notin \langle\lt_>(g_1),\ldots,\lt_>(g_k)\rangle$, unless $r=0$.
         \end{description}
         Moreover, the algorithm requires only a finite number of recursions.}
       \STATE Compute \begin{displaymath}
         (a,q,f):=\pRed(f,>).
       \end{displaymath}
       \IF{$f\neq 0$ and $\exists\;i\;:\;\lt_>(g_i)\;|\;\lt_>(f)$}
       \STATE Set $D:=\{ g_i\in G \mid \lt_>(g_i) \text{ divides } \lt_>(f) \}$.
       \STATE Pick $g_j\in D$ with minimal ecart. %$\ecart_>(g)=\min\{\ecart_>(g')\mid g'\in D\}$.
       \IF{$e:= \ecart_>(g_j) -  \ecart_>(f)>0$}
       \STATE Compute
       \begin{displaymath}
         ((Q_1',\ldots,Q_k'),R'):=\SHDDwR(x_0^e\cdot f^h,(\lt_>(g_1^h),\ldots,\lt_>(g_k^h)),>_h).
       \end{displaymath}
       \STATE Set $f':=(x_0^e\cdot f^h - \sum_{i=1}^k Q_i'\cdot g_i^h)^d$.
       \STATE Compute \begin{displaymath}
         (u'',(q_1'',\ldots,q_{k+1}''),r):=\SDwR(f',(g_1,\ldots,g_k,f),>).
       \end{displaymath}
       \STATE Set $q_i:= q_i''+u''\cdot Q_i'^d$, $i=1,\ldots,k$.
       \STATE Set $u:=u''-q_{k+1}''$.
       \ELSE
       \STATE Compute \begin{displaymath}
         ((Q_1',\ldots,Q_k'),R'):=\SHDDwR(f^h,(g_1^h,\ldots,g_k^h),>_h).
       \end{displaymath}
       \STATE Compute \begin{displaymath}
         (u,(q_1'',\ldots,q_k''),r):=\SDwR((R')^d,(g_1,\ldots,g_k),>).
       \end{displaymath}
       \STATE Set $q_i:=q_i''+u\cdot Q_i'^d$, $i=1,\ldots,k$.
       \ENDIF
       \ELSE
       \STATE Set $(u,(q_1,\ldots,q_k),r):=(1,(0,\ldots,0),f)$.
       \ENDIF
       \RETURN{$(a\cdot u,(q_1+q,q_2,\ldots,q_k),r)$.}
     \end{algorithmic}
   \end{algorithm}
   \begin{proof}
     Note first, that after Step $1$ the new polynomial $f$ has leading
     coefficient $1$, its leading monomial is less than or equal to
     that of the original $f$ and the same holds for the leading
     monomial $\lm_>(q)=\lm_>(q\cdot g_1)$.

     We then should keep in mind that, as in Algorithm~\ref{alg:tropHDDwR},
     the condition
     \begin{displaymath}
       \lt_>(f)\in\langle \lt_>(g_1),\ldots,\lt_>(g_k)\rangle
     \end{displaymath}
     is equivalent to
     \begin{displaymath}
       \exists\;i\;:\;\lt_>(g_i)\;|\;\lt_>(f).
     \end{displaymath}

     \emph{Finiteness of recursions:} For sake of clarity, label all
     the objects appearing in the $\nu$-th recursion
     step by a subscript $\nu$. For example the ecart $e_\nu\in\NN$,
     the element $f_\nu\in\ZZ[x]$ and the subset
     $G_\nu\subseteq\ZZ[x]$.

     Since $G_1^h\subseteq G_2^h\subseteq G_3^h\subseteq \hdots$, we
     have an ascending chain of leading ideals in $\ZZ[x_h]$, which
     eventually stabilizes unless the algorithm terminates beforehand
     \begin{displaymath}
       \LT_{>_h}(G_1^h)\subseteq  \LT_{>_h}(G_2^h)\subseteq \ldots \subseteq \LT_{>_h}(G_N^h) = \LT_{>_h}(G_{N+1}^h) = \ldots.
     \end{displaymath}

     Assume $e_N>0$. Then we'd have $f_N\in G_{N+1}$, and thus
     \begin{displaymath}
       \lt_{>_h}(f_N^h)\in \lt_{>_h}(G_{N+1}^h)=\lt_{>_h}(G_N^h).
     \end{displaymath}
     To put it differently, we'd have a $g^h\in G_N^h$ such that
     \begin{displaymath}
       \lt_{>_h}(g^h)\;\mid\;\lt_{>_h}(f_N^h),
     \end{displaymath}
     which by Remark \ref{re:homogenization} (5) would imply that
     \begin{displaymath}
       \lt_>(g)\;\mid\;\lt_>(f_N)
       \text{ and }\ecart_>(g)\leq\ecart_>(f_N).
     \end{displaymath}
     This contradicts our assumption
     \begin{displaymath}
       e_N = \min\{ \ecart_>(g) \mid g \in D_N \} - \ecart_>(f_N) \overset{!}{>} 0. \quad \lightning
     \end{displaymath}
     Therefore we have $e_N\leq 0$. By induction we conclude that
     $e_\nu\leq 0$ for all $\nu\geq N$, i.e. that we will exclusively
     run through steps 13-15 of the ``else'' case from the $N$-th
     recursion step onwards.

     By the properties of HDDwR we know that in particular
     \begin{equation}\label{eq:SDwR:2}
       \lt_{>_h}(R_N')\notin \LT_>(G_N^h).
     \end{equation}
     Now assume that the recursions would not stop with the next recursion.
     That means there exists a $g\in G_N=G_{N+1}$ such that
     \begin{displaymath}
       \lt_>(g)\;\mid\;\lt_>(f_{N+1})=\lt_>((R_N')^d),
     \end{displaymath}
     and because of $e_{N+1}\leq 0$ also
     \begin{displaymath}
       \ecart(g)\leq\ecart(f_{N+1})=\ecart((R_N')^d).
     \end{displaymath}
     It then  follows from Remark \ref{re:homogenization}
     (6) that
     \begin{displaymath}
       \lt_{>_h}(g^h)\;\mid\;\lt_{>_h}(R_N'),
     \end{displaymath}
     in contradiction to \eqref{eq:SDwR:2}.
     Hence the algorithm terminates after the $N+1$-st recursion step.

     \emph{Correctness:}
     In what follows we will denote by $f$ the original polynomial and by
     $\tilde{f}$ the polynomial $f$ after Step $1$. Moreover, we
     recall that
     \begin{equation}\label{eq:SDwR:3}
       a\cdot f=q\cdot g_1+\tilde{f}
     \end{equation}
     with $\lm_>(f)\geq \lm_>(q)=\lm_>(q\cdot g_1)$.

     We make an induction on the number of recursions, say $N\in\NN$.
     If $N=1$ then either $\tilde{f}=0$ or $\lt_>(\tilde{f})$ is not
     divisible by any $\lt_>(g_i)$, and in both cases
     \begin{displaymath}
       1\cdot \tilde{f}=0\cdot g_1+\hdots+0\cdot g_k+\tilde{f}
     \end{displaymath}
     satisfies (ID1) and (ID2), and thus by \eqref{eq:SDwR:3} so does
     \begin{displaymath}
       a\cdot f=q\cdot g_1+0\cdot g_2+\ldots 0\cdot g_k+\tilde{f}.
     \end{displaymath}

     So suppose $N>1$ and consider the first recursion step.
     If $e \leq 0$, then by the properties of HDDwR the representation
     \begin{displaymath}
       \tilde{f}^h=Q_1'\cdot g_1^h+\hdots+ Q_k'\cdot g_k^h+R'
     \end{displaymath}
     satisfies (DD1), (DD2) and (DDH). (DD1) and (DD2) imply (ID1), which means that for each $i=1,\hdots,k$ we have
     \begin{flushleft}
       $x_0^{\ecart_>(\tilde{f})}\cdot\lm_>(\tilde{f})=\lm_{>_h}(\tilde{f}^h)\overset{(\text{ID1})}{\geq_h} \lm_{>_h}(Q_i')\cdot\lm_{>_h}(g_i^h)=\ldots$
     \end{flushleft}
     \begin{flushright}
       $\ldots=x_0^{a_i+\ecart_>(g_i)}\cdot\lm_>(Q_i'^d)\cdot\lm_>(g_i)$
     \end{flushright}
     for some $a_i\geq 0$. Since $\tilde{f}^h$ and $Q_i'\cdot g_i^h$ are both $x_h$-homogeneous of the same $x_h$-degree
     by (DDH), the definition of the homogenized ordering $>_h$ implies
     \begin{equation}
       \label{eq:dwr11}
       \lm_>(\tilde{f})\geq\lm_>(Q_i'^d)\cdot\lm_>(g_i) \text{ for all } i=1,\ldots,k.
     \end{equation}
     Moreover, by induction the representation $u\cdot R'^d=q_1''\cdot g_1+\ldots+q_k''\cdot g_k+r$
     satisfies (ID1), (ID2) and $p\nmid\lc_>(u)=\lt_>(u)$, the first implying that
     \begin{equation}
       \label{eq:dwr12}
       \lm_>(\tilde{f})\overset{(\ref{eq:dwr11})}{\geq}\lm_>\underbrace{\left(\tilde{f}-\sum_{i=1}^kQ_i'^d\cdot g_i\right)}_{=R'^d}\overset{\text{(ID1)}}{\geq}\lm_>(q_i''\cdot g_i).
     \end{equation}
     Therefore, the representation
     \begin{displaymath}
       u\cdot \tilde{f}=\sum_{i=1}^k(q_i''+u\cdot Q_i'^d)\cdot g_i+r
     \end{displaymath}
     satisfies (ID1) by (\ref{eq:dwr11}), (\ref{eq:dwr12}), $p\nmid
     \lc_>(u)=\lt_>(u)$ and (ID2) by induction, and hence by
     \eqref{eq:SDwR:3} so does the
     representation
     \begin{displaymath}
       a\cdot u\cdot f=(q_1''+u\cdot Q_i'^d+q)\cdot g_1+\sum_{i=2}^k(q_i''+u\cdot Q_i'^d)\cdot g_i+r.
     \end{displaymath}

     Similarly, if $e>0$, then by the properties of HDDwR the representation
     \begin{displaymath}
       x_0^e\cdot \tilde{f}^h=Q_1'\cdot\lt_{>_h}(g_1^h)+\ldots+Q_k'\cdot\lt_{>_h}(g_k^h)+R'
     \end{displaymath}
     satisfies (DD1), (DD2) and (DDH). (DD1) and (DD2) imply (ID1), which means that for each $i=1,\ldots,k$ we have
     \begin{flushleft}
       $x_0^{e+\ecart_>(\tilde{f})}\cdot\lm_>(\tilde{f})=\lm_{>_h}(x_0^e\cdot \tilde{f}^h) \geq \ldots$
     \end{flushleft}
     \begin{flushright}
       $\ldots\geq\lm_{>_h}(Q_i')\cdot\lm_{>_h}(\lt_{>_h}(g_i^h))=x_0^{a_i+\ecart_>(g_i)}\cdot\lm_>(Q_i'^d)\cdot\lm_>(g_i),$
     \end{flushright}
     for some $a_i\geq 0$.
     Since $x_0^e\cdot \tilde{f}^h$ and $Q_i'\cdot \lt_{>_h}(g_i^h)$ are both
     $x_h$-homogeneous of the same $x_h$-degree by (DDH), the
     definition of the homogenized ordering $>_h$ implies
     \begin{equation}
       \label{eq:dwr21}
       \lm_>(\tilde{f})\geq\lm_>(Q_i'^d)\cdot\lm_>(g_i).
     \end{equation}
     Moreover, by induction the representation
     $u''\cdot \tilde{f}'=\sum_{i=1}^kq_i''\cdot g_i+q_{k+1}''\cdot \tilde{f}+r $
     satisfies (ID1), (ID2), $p\nmid \lc_>(u'')=\lt_>(u'')$ with the first implying that
     \begin{equation}
       \label{eq:dwr22}
       \lm_>(\tilde{f})\overset{(\ref{eq:dwr21})}{\geq}\underbrace{\lm_>\left(\tilde{f}-\sum_{i=1}^kQ_i'^d\cdot g_i\right)}_{=\lm_>(R'^d)}\overset{\text{(ID1)}}{\geq} \lm_>(q_i''\cdot g_i).
     \end{equation}
     Therefore, the representation
     \begin{displaymath}
       u\cdot \tilde{f}=\sum_{i=1}^k(q_i''+u''\cdot Q_i'^d)\cdot g_i+r, \text{ with } u=u''-q_{k+1}''
     \end{displaymath}
     satisfies (ID1) by (\ref{eq:dwr21}), (\ref{eq:dwr22}), $p\nmid \lc_>(u'')=\lt_>(u'')$ and (ID2) by induction.

     To see that $\lt_>(u)=\lt_>(u'')$ and hence
     $p\nmid\lc_>(u)=\lt_>(u)$, observe that
     \begin{displaymath}
       \lt_{>_h}(x_0^e\cdot \tilde{f}^h)\in \langle\lt_>(g_1^h),\ldots,\lt_>(g_k^h)\rangle,
     \end{displaymath}
     which is why
     \begin{displaymath}
       \lm_>(\tilde{f})=\lm_{>_h}(x_0^e \cdot \tilde{f}^h)^d>\lm_{>_h}\left(x_0^e\cdot \tilde{f}^h - \sum_{i=1}^k Q_i'\cdot g_i^h\right)^d=\lm_>(\tilde{f}').
     \end{displaymath}
     Thus $\lm_>(\tilde{f})>\lm_>(\tilde{f}')\geq\lm_>(q_{k+1}'')\cdot\lm_>(\tilde{f})$, which implies $\lm_>(q_{k+1}'')<1$
     and thus $\lt_>(u)=\lt_>(u'')$.
   \end{proof}

   \begin{remark}
     The representation
     \begin{equation}\label{eq:remSDwR:1}
       u\cdot f=q_1\cdot g_1+\hdots+q_k\cdot g_k+r
     \end{equation}
     that we compute in Algorithm~\ref{alg:tropDwR} is actually not a
     standard representation in the sense that we defined, even though
     it satisfies (ID1) and (ID2). The reason is, that we replaced the
     condition
     \begin{displaymath}
       \lt_>(u)=1
     \end{displaymath}
     by the weaker condition
     \begin{displaymath}
       p\nmid\lc_>(u)\text{ and } \lm_>(u)=1.
     \end{displaymath}
     However, if $p$ does not divide the integer $\lc_>(u)$ then this
     number is invertible in the ring of $p$-adic numbers
     \begin{displaymath}
       \ZZ\llbracket t\rrbracket/\langle p-t\rangle\cong\ZZ_p,
     \end{displaymath}
     which implies that there are power series $g,h\in\ZZ\llbracket
     t\rrbracket$ such that
     \begin{displaymath}
       g\cdot\lc_>(u)=1+h\cdot (p-t).
     \end{displaymath}
     Replacing in the above representation $u$ by $g\cdot u$, $r$ by
     $g\cdot r$,  $q_1$ by
     $g\cdot q_1-h$ and $q_i$ by $g\cdot q_i$ for $i=2,\ldots,k$ we
     get a standard representation with coefficients in $\ZZ\llbracket
     t\rrbracket[x]$. The representation is thus good enough for our
     purposes.

     We, actually, could even easily turn \eqref{eq:remSDwR:1} into a
     polynomial standard representation as follows. If $a,b\in\ZZ$ with
     \begin{displaymath}
       a\cdot\lc_>(u)+b\cdot p=1
     \end{displaymath}
     and if
     \begin{displaymath}
       b=\sum_{j=0}^l c_j\cdot p^j
     \end{displaymath}
     is the $p$-adic expansion of $b$, then
     \begin{displaymath}
       a\cdot \lc_>(u)=1-\sum_{j=1}^{l+1}c_{j-1}\cdot p^j
       =
       1-\sum_{j=1}^{l+1}c_{j-1}\cdot t^j+h\cdot (p-t)
     \end{displaymath}
     for some polynomial $h\in\ZZ[t]$. With
     \begin{displaymath}
       v=1-\sum_{j=1}^{l+1}c_{j-1}\cdot t^j+\tail(u)
     \end{displaymath}
     and multiplying
     \eqref{eq:remSDwR:1} by $a$ we thus get
     \begin{displaymath}
       (v+h\cdot (p-t))\cdot f=\sum_{i=1}^k a\cdot q_i+a\cdot r
     \end{displaymath}
     or equivalently
     \begin{displaymath}
       v\cdot f=(a\cdot q_1-h\cdot f)\cdot g_1+\sum_{i=2}^k
       a\cdot q_i\cdot g_i+a\cdot r,
     \end{displaymath}
     which is a standard representation with $\lc_>(v)=1$ and
     $v,q_1,\ldots,q_k,r\in\ZZ[t,x]$.
%
%      We should like to remark that in Algorithm~\ref{alg:tropDwR}
%      \begin{displaymath}
%        \LT_{>_h}(g_1)=\LT_{>_h}(p x_0-t)=p x_0
%      \end{displaymath}
%      and that thus in Step $6$ every term in $x_0^e\cdot f^h$ whose leading
%      coefficient is divisible by $p$ will be completely reduced, so
%      that in $f'$ no coefficient is divisible by $p$.
   \end{remark}

   With this division with remainder at hand, we can formulate the
   standard basis algorithm for this special setting. The proof works
   as in Algorithm~\ref{alg:standardBases}.

   \begin{algorithm}[standard basis algorithm --- special case]\label{alg:standardBasesSpecial} \skipalgorithm
     \begin{algorithmic}[1]
       \REQUIRE $(G,>)$, where $G=(g_1,\ldots,g_k)$ be a $k$-tuple of
       elements in $\ZZ[t,x]$ with $g_1=p-t$
       and $>$ a
       $t$-local monomial ordering on $\Mon(t,x)$.
       \ENSURE $G'=(g'_1,\ldots,g'_l)$ a standard basis of $\langle
       G\rangle\unlhd\ZZ\llbracket t\rrbracket[x]$ with respect to $>$
       such that $g_1'=p-t$ and $\lc_>(g'_i)=1$ for $i=2,\ldots,l$.
       \FOR{$i=2,\ldots,k$}
       \STATE Compute $(a,q,r):=\pRed(g_i,>)$.
       \STATE Set $g_i:=r$.
       \ENDFOR
       \STATE Initialize a pair-set, $P:=\{(g_i,g_j)\mid i<j\}$.
       \WHILE{$P\neq\emptyset$}
       \STATE Pick $(g_i,g_j)\in P$.
       \STATE Set $P:=P\setminus\{(g_i,g_j)\}$.
       \STATE Compute
       \begin{displaymath}
         (u,(q_1,\ldots,q_k),r):=\SDwR_>(\spoly(g_i,g_j), G,>),
       \end{displaymath}
       where
       \begin{align*}
         &\spoly(g_i,g_j) \\
         &\qquad =\frac{\lcm(\lt_>(g_i),\lt_>(g_j))}{\lt_>(g_i)}\cdot g_i - \frac{\lcm(\lt_>(g_i),\lt_>(g_j))}{\lt_>(g_j)}\cdot g_j
       \end{align*}
       and
       \begin{align*}
         &\lcm(\lt_>(g_i),\lt_>(g_j)) \\
         &\qquad = \lcm(\lc_>(g_i),\lc_>(g_j))\cdot \lcm(\lm_>(g_i),\lm_>(g_j)).
       \end{align*}
       \IF{$r\neq 0$}
       \STATE Compute $(a,q,r):=\pRed(r,>)$.
       \STATE Extend the pair-set, $P:=P\cup\{(g,r)\mid g\in G\}$.
       \STATE Set $G:=G\cup\{r\}$.
       \ENDIF
       \ENDWHILE
       \RETURN{$G':=G$.}
     \end{algorithmic}
   \end{algorithm}

   \begin{remark}
     We would like to remark that the standard basis elements
     will be $x$-homogeneous if the input
     was so.
   \end{remark}

}

%%%%%%%%%%%%%%%%%%%%%%%%%%%%%%%%%%%%%%%%%%%%%%%%%%%%%%%%%%%%%%%%%%%%%%%%%%%%%%%%%%

   \section{Reduced standard bases}

   In this rather short section we recall the notion of a reduced
   standard basis and show what problems we run into when allowing
   base rings that are not fields and local orderings. Reduced
   standard bases play a very important role in the computation of
   Gr\"obner fans and tropical varieties. Since they turn not to be
   computationally feasible in our setting, we will replace them
   by a weaker notion that is
   good enough for the computation of Gr\"obner fans and tropical
   varieties.

   \begin{definition}
     Let $G,H\subseteq \Rtx^s$ be two finite subsets. Given a
     $t$-local monomial ordering $>$ on $\Mon^s(t,x)$, we call $G$
     \emph{reduced} with respect to $H$, if, for all $g\in G$, no term
     of $\tail_>(g)$ lies in $\LT_>(H)$.

     And we simply call $G$ \emph{reduced}, if it is reduced with
     respect to itself and minimal in the sense that no proper subset
     $G'\subsetneq G$ is sufficient to generate its leading module,
     i.e. $\lt_>(G')\subsetneq \lt_>(G)$.

     Observe that we forego any kind of normalization of the leading
     coefficients that is normally done in polynomial rings over
     ground fields.
   \end{definition}

   If our module is generated by $x$-homogeneous elements, it is not
   hard to show that reduced standard bases exist. Given an
   $x$-homogeneous standard basis, one can pursue a strategy similar
   to the classical reduction algorithm based on repeated tail
   reduction. Lemma \ref{lem:markwig2} guarantees its convergence in
   the $\langle t\rangle$-adic topology.

   \begin{algorithm}[reduction algorithm]\label{alg:totalReduction} \skipalgorithm
     \begin{algorithmic}[1]
       \REQUIRE $(G,>)$, where $G=\{g_1,\ldots,g_k\}$ is a minimal $x$-homogeneous standard basis of $M\leq \Rtx^s$ with respect to the weighted ordering $>=>_w$ with $w\in\RR_{<0}^m\times\RR^{n+s}$.
       \ENSURE $G'=\{g_1',\ldots,g_k'\}$ an $x$-homogeneous reduced standard basis of $M$ with respect to $>$ such that $\lm_{>}(g_i')=\lm_{>}(g_i)$.
       \FOR{$i=1,\ldots,k$}
       \STATE Set $g_i':=g_i$.
       \STATE Create a working list
       \begin{displaymath}
         L:=\{ p \in \Rtx^s \mid p \text{ term of } g_i',\,\LM_>(g_i') > p\},
       \end{displaymath}
       \WHILE{$L\neq \emptyset$}
       \STATE Pick $p\in L$ with $\lm_>(p)$ maximal.
       \STATE Set $L:=L\setminus\{p\}$.
       \IF{$p\in \lt_>(M)$}
       \STATE Compute homogeneous division with remainder
       \begin{displaymath}
         ((q_1,\ldots,q_k),r) = \HDDwR(p,(g_1,\ldots,g_k),>).
       \end{displaymath}
       \STATE Set $g_i':=g_i'-(q_1\cdot g_1+\ldots+q_k\cdot g_k)$.
       \STATE Update the working list
       \begin{displaymath}
         L:=\{ p' \in \Rtx^s \mid p' \text{ term of } g_i,\, \lm_>(p)> \lm_>(p') \}.
       \end{displaymath}
       \ENDIF
       \ENDWHILE
       \ENDFOR
       \RETURN{$\{g_1',\ldots,g_k'\}$}
     \end{algorithmic}
   \end{algorithm}
   \begin{proof}
     Pick an $i=1,\ldots,k$. Labelling all objects occurring in the
     $\nu$-the recurring step by a subscript $\nu$, we have a strictly
     decreasing sequence
     \begin{displaymath}
       \lm_>(p_1)>\lm_>(p_2)>\lm_>(p_3)>\ldots\, .
     \end{displaymath}
     And since $\lm_>(p_\nu)\geq \lm_>(q_{j,\nu}\cdot g_j)$ for all
     $j=1,\ldots,k$, the sequence $(q_{j,\nu}\cdot g_j)_{\nu\in\NN}$
     must also converge in the $\langle t\rangle$-adic topology
     together with $(p_\nu)_{\nu\in\NN}$.
     In particular, the element $g_i'=g_i - \sum_{\nu=0}^\infty \sum_{j=1}^k q_{j,\nu}\cdot g_j$ in our output exists.

     Also, while setting $g_{i,\nu+1}'=g_{i,\nu}'-(q_{1,\nu}\cdot
     g_1+\ldots+q_{k,\nu}\cdot g_k)$ apart from the term $p_\nu$
     cancelling, the terms changed are all strictly smaller than
     $p$. Hence for any term $p$ of $g_i'$, $p\neq\LT_>(g_i)$, there
     is a recursion step in which it is picked. Because $p$ is not
     cancelled during the step, we have $p\notin\LT_>(M)$. Therefore
     no term of $g_i'$ apart from its leading term lies in
     $\LT_>(M)$.
   \end{proof}

   One nice property of reduced standard bases, that is repeatedly
   used in the established theory of Gr\"obner fans of polynomial
   ideals over a ground field, is their uniqueness up to
   multiplication by units. In fact, this property does not change
   even if we add power series into the mix.

   \begin{lemma}\label{lem:reducedStandardBasisUnique}
     Let $R$ be a field and let $M\leq\Rtx^s$ or $M\leq\Rtx^s_>$ be a
     module generated by $x$-homogeneous elements. Then $M$ has a
     unique monic, reduced standard basis.
   \end{lemma}
   \begin{proof}
     Because $R$ is a field, we have $\LT_>(M)=\LM_>(M)$ and since
     $\LM_>(M)$ has a unique minimal generating system consisting of
     monomials, let's call it $A$, so does $\LT_>(M)$.

     Let $G=\{g_1,\ldots,g_k\}$ be a monic, reduced standard bases of
     $M$. Observe that the leading terms of $G$ form a standard basis
     of the leading module of $M$. That means each $a\in A\subseteq
     \LT_>(M)$ can be expressed with a standard representation of the
     leading terms of $G$,
     \begin{displaymath}
       a = q_1\cdot \LT_>(g_1)+\ldots+ q_k\cdot \LT_>(g_k).
     \end{displaymath}
     Since there is no cancellation of higher terms in the standard
     representation, there must exist an $i=1\ldots,k$ with
     $a=\LM_>(q_i\cdot g_i)$. This implies $\LM_>(g_i)=a$ because $a$
     wouldn't be a minimal generator of $\LM_>(M)$ otherwise. And
     because $G$ is monic, $\LT_>(g_i)=a$.

     Therefore, given a reduced standard basis $G$, we see that for
     any minimal generator $a\in A$ there exists an element $g\in G$
     with $\lm_>(g)=a$. And since reduced standard bases are minimal
     themselves, it means that there is exactly one element $g\in G$
     per minimal generator $a\in A$.

     Now let $G$ and $H$ be two different reduced standard basis of
     $M$. Let $a\in A$ and let $g\in G$, $h\in H$ be the basis element
     with leading monomial $a$. If $g-h\neq 0$, then $g-h\in M$ must
     have a non-zero leading monomial which lies in
     $\LM_>(M)$. However, that monomial also has to occur in either
     $g$ and $h$, and since $R$ is a field the term with that monomial
     has to lie in $\LT_>(M)=\LM_>(M)$, contradicting that $G$ and $H$
     were reduced.
   \end{proof}

   However, it can easily be seen that this does not hold over rings.

   \begin{example}
     Consider the ring $\ZZ[x,y]$ and the degree lexicographical ordering $>$, i.e.
     \begin{align*}
       &x^{a_1} y^{a_2} > x^{b_1} y^{b_2} \quad :\Longleftrightarrow \\
       & \qquad a_1+a_2> b_1+b_2 \text{ or} \\
       & \qquad a_1+a_2=b_1+b_2 \text{ and } (a_1,a_2)>(b_1,b_2) \text{ lexicographically in } \RR^2.
     \end{align*}
     Consider the following ideal and its leading ideal:
     \begin{displaymath}
       I:=\langle 2x^2y+1, 3xy^2+1\rangle \text{ and } \LT_>(I)=\langle 2x, 9y^3, xy^2 \rangle.
     \end{displaymath}
     Two possible standard bases, both reduced, are (leading terms highlighted)
     \begin{center}
       \begin{tikzpicture}
         \matrix (m) [matrix of math nodes, row sep=1em, column sep=0em]
         { G_1=\{&\underline{2x}-3y,&\underline{9y}^3+2,&\underline{xy}^2+3y^3+1&\}, \\
           G_2=\{&\underline{2x}-3y,&\underline{9y}^3+2,&\underline{xy}^2-6y^3-1&\}. \\
         };
         \draw[draw opacity=0] (m-2-2) -- node[sloped] {$=$} (m-1-2);
         \draw[draw opacity=0] (m-2-3) -- node[sloped] {$=$} (m-1-3);
         \draw[draw opacity=0] (m-2-4) -- node[sloped] {$\neq$} (m-1-4);
       \end{tikzpicture}
     \end{center}
     Hence, unlike their classical counterparts over ground fields,
     reduced standard bases over ground rings are not unique up to
     multiplication with units. The key problem is that leading
     modules are not necessarily saturated with respect to the ground
     ring. This allowed the third basis element to have terms with
     monomials in $\LM_>(I)$, to which we could add a constant
     multiple of the second basis element without changing it being
     reduced.
   \end{example}

   Note also, that even if the base ring is a field and the ideal is
   generated by a polynomial, the reduced standard basis might contain
   power series. This is a well known fact when dealing with local
   orderings.

\ifx\langform\ja
   \begin{example}\label{ex:reductionInfiniteness}
     Consider the principal ideal generated by the element
     $g=x+y+tx\in \QQ\llbracket t\rrbracket [x,y]$ and the monomial
     ordering $>_w$ with weight vector $w=(-1,1,1)$ and $>$ the
     lexicographical ordering with $x>y>1>t$ as tiebreaker. Then
     $\{g\}$ is a standard basis and one can show that it converges to
     $g'=x+\sum_{i=0}^\infty (-1)^i\cdot t^iy$ in its reduction
     process.
     \begin{figure}[h]
       \centering
       \begin{tikzpicture}
         \matrix (m) [matrix of math nodes, row sep=2em, column sep=-0.5em, column 1/.style={anchor=base east}]
         { x+y+&\underline{tx} \\
           x+y-ty-&\underline{t^2x} \\
           x+y-ty+t^2y+&\underline{t^3x} \\
           & \vdots \\
           & \phantom{t^iy} \\ };
         \draw[->,shorten >=3pt,shorten <=3pt,font=\scriptsize] (m-1-2) -- node[right] {$-t\cdot g$} (m-2-2);
         \draw[->,shorten >=3pt,shorten <=3pt,font=\scriptsize] (m-2-2) -- node[right] {$+t^2\cdot g$} (m-3-2);
         \draw[->,shorten >=3pt,shorten <=3pt,font=\scriptsize] (m-3-2) -- (m-4-2);
         \draw[->,shorten >=3pt,shorten <=3pt,font=\scriptsize] (m-4-2) -- (m-5-2);
         \node[anchor=east] at (m-5-2.east) {$x+\sum_{i=0}^\infty (-1)^i\cdot t^iy$};
       \end{tikzpicture}
       \caption{reduction of $tx+t^2x+y$}
       \label{fig:exampleTotalReduction}
     \end{figure}

     Since the reduced standard basis is unique, this implies that $I$
     has no reduced standard basis consisting of polynomials, even
     though $I$ is generated by a polynomial itself. Consequently,
     this means that the reduced standard bases which play a central
     role in the established Gr\"obner fan theory are useless in our
     case from a practical perspective.
   \end{example}
\fi

   In
   \cite{MR15a} we will weaken the notion of reducedness, and we will
   show that this weakened version can be computed and is strong
   enough to compute Gr\"obner fans (see \cite{MR15a}) and tropical
   varieties (see \cite{MR15b}).

\newcommand{\etalchar}[1]{$^{#1}$}
\def\cprime{$'$}
\providecommand{\bysame}{\leavevmode\hbox to3em{\hrulefill}\thinspace}

\end{document}